\documentclass[review,10pt]{elsarticle} 

\usepackage[utf8]{inputenc}
\usepackage{hyperref}
\usepackage{amssymb, amsmath, amsthm}
\usepackage{amsfonts} 
\usepackage{color}
\usepackage{mathrsfs}
\usepackage{comment}
\usepackage{dcolumn}
\usepackage{bm} 
\usepackage[hang,small,bf]{caption}
\usepackage[subrefformat=parens]{subcaption}
\usepackage{algorithmic}
\usepackage{algorithm}
\numberwithin{equation}{section}
\newcommand{\e}{\varepsilon}
\renewcommand{\d}{\mathrm{d}}
\newcommand{\R}{\mathbb R}
\newcommand{\N}{\mathbb N}
\newcommand{\dv}{{\rm{div}}}

\newtheorem{thm}{Theorem}[section]
\newtheorem{lem}[thm]{Lemma}
\newtheorem{rmk}[thm]{Remark}
\newtheorem{prop}[thm]{Proposition}
\newtheorem{defi}[thm]{Definition}

\newtheorem{ex}[thm]{Example}

\newcommand{\Hm}[1]{\leavevmode{\marginpar{\tiny%
$\hbox to 0mm{\hspace*{-0.5mm}$\leftarrow$\hss}%
\vcenter{\vrule depth 0.1mm height 0.1mm width \the\marginparwidth}%
\hbox to 0mm{\hss$\rightarrow$\hspace*{-0.5mm}}$\\\relax\raggedright
#1}}}

\begin{document}
\begin{frontmatter}

\title{Optimal design problem with thermal radiation \tnoteref{mytitlenote}
}

\tnotetext[mytitlenote]{
K.K.~is partially supported by JSPS KAKENHI Grant Numbers JP24K16955 and JP24KJ0010.
K.M.~is partially supported by Mizuho Foundation for the Promotion of Sciences and JSPS KAKENHI Grant Numbers JP24K17191 and JP23H03413.
T.O.~is partially supported by JSPS KAKENHI Grant Numbers
JP22K20331 and JP23K12997.
This work is also supported by the Research Institute for Mathematical Sciences, an International Joint Usage/Research Center located in Kyoto University.
}
\author[mainaddress]{Kosuke Kita}
\ead{kosuke.kita.c3@tohoku.ac.jp}
\author[mymainaddress]{Kei Matsushima}
\ead{matsushima@mech.t.u-tokyo.ac.jp}
\author[mymainaddress2]{Tomoyuki Oka\corref{mycorrespondingauthor}}
\cortext[mycorrespondingauthor]{Corresponding author}
\ead{t-oka@fit.ac.jp}
\address[mainaddress]{Graduate School of Science, Tohoku University, Aoba-ku, Sendai 980-8578, Japan}
\address[mymainaddress]{Graduate School of Engineering, The University of Tokyo, Bunkyo-ku, Tokyo 113-8656,Japan}
\address[mymainaddress2]{Faculty of Engineering, Fukuoka Institute of Technology, Higashi-ku, Fukuoka 811-0295, Japan
}

\begin{keyword} 
Optimal design problem, 
Topology optimization, 
Nonlinear boundary condition, 
Thermal radiation, 
Homogenization, 
Level set method
\MSC[2020] {\emph{Primary}: 80M50; \emph{Secondary}: 35J65, 80M40}
\end{keyword}


\begin{abstract}
This paper is concerned with configurations of two-material thermal conductors that minimize the Dirichlet energy for steady-state diffusion equations with nonlinear boundary conditions described mainly by maximal monotone operators. 
To find such configurations, a homogenization theorem will be proved and applied to an existence theorem for minimizers of a relaxation problem whose minimum value is equivalent to an original design problem. 
As a typical example of nonlinear boundary conditions, thermal radiation boundary conditions will be the focus, and then the Fr\'echet derivative of the Dirichlet energy will be derived, which is used to estimate the minimum value.
Since optimal configurations of the relaxation problem involve the so-called grayscale domains that do not make sense in general, a perimeter constraint problem via the positive part of the level set function will be introduced as an approximation problem to avoid such domains, and moreover, the existence theorem for minimizers of the perimeter constraint problem will be proved.
In particular, it will also be proved that 
the limit of minimizers for the approximation problem becomes that of the relaxation problem in a specific case, and then candidates for minimizers of the approximation problem will be constructed by employing time-discrete versions of nonlinear diffusion equations. 
In this paper, it will be shown that optimized configurations deeply depend on force terms 
as a characteristic of nonlinear problems and will also be applied to real physical problems.
\end{abstract}

\end{frontmatter}

\section{Introduction} 
Transfer of thermal energy can be classified into conduction, convection, and radiation. Generally, the first two are usually treated as linear Partial Differential Equations (PDEs)  as long as thermal conductivity, specific heat and heat source are independent of temperature. In the case of the steady-state, many problems determining distributions of composite materials such that energies (or objective functionals) are minimized have been studied in various fields such as mathematics, physics, engineering and computer science as \emph{optimal design problem} (or \emph{shape and topology optimization problem}).
Material distributions of two-material composites are represented by employing characteristic functions, and therefore, the optimal design problem is described by a minimization problem of functionals with respect to the characteristic function. 
To guarantee the existence of minimizers, \emph{homogenization theory} plays a crucial role since the minimizing sequence of the characteristic functions usually oscillates (see, e.g.,~\cite{Murat-Tartar}).
Although there are no minimizers in general, it has been proved that there exists a pair of volume fractions (or densities) and homogenized coefficients that achieve the minimum value (see, e.g.,~\cite{A02,DLM23}). 
Hence it is a crucial issue to construct volume fractions that are similar to characteristic functions and give values close to the minimum, and therefore, various numerical techniques that do not involve intermediate sets have been devised (see e.g.,~\cite{AJT02,AJT04,AA06,Td2, LZ11}).
As for the unsteady-state, the existence of minimizers for a true relaxation problem has been proved in \cite{MPP08} along with the result in \cite{BFF92} and clarified the long-time behavior of optimal configurations concerning volume fraction in \cite{AMP10}. 

On the other hand, thermal radiation depends on temperature due to Stefan–Boltzmann law (cf.~\cite{LV89}) and is stated as nonlinear PDEs (in particular, it is formulated by
PDEs with nonlinear boundary conditions). From a physical point of view, the Stefan--Boltzmann law implies that the intensity of electromagnetic radiation from a black body is nonlinearly dependent on its surface temperature. Thermal radiation is especially important in high-temperature environments or thermally isolated systems in terms of conduction and convection, e.g.,~satellites and spacecrafts in outer space \cite{gilmore2002spacecraft}. 

Such potential applications motivate us to develop an optimal design method for nonlinear PDEs. As shape and topology optimization is widely used for enhancing static and dynamic properties of solid deformation, much effort has been devoted to shape and topology optimization for nonlinear elasticity, e.g., large deformation \cite{jog1996distributed}, hyper-elasticity \cite{bruns2001topology}, and elasto-plasticity \cite{yuge1995optimization}. Another important subject is the fluid dynamics with the aim of designing, e.g., aircraft wings \cite{maute2004conceptual}, microfluidic systems \cite{chen2016topology}, and heat-dissipating structures with natural convection \cite{yoon2010topological}. See also \cite{KS86,JCD22,HMP21,ZHW18} for optimal design or shape optimization problems of other nonlinear PDEs.

    Some earlier studies have presented shape-optimized designs of heat radiators. For example, Dems and Korycki calculated a shape derivative associated with steady radiative heat transfer and optimized some parameters representing the boundary of a heat radiator \cite{dems2005sensitivity}. Transient heat transfer with radiation was also studied in \cite{korycki2006sensitivity}. Recent work by Liu and Hasegawa developed a level set-based method for shape optimization in radiative heat transfer \cite{liu2023adjoint}. However, rigorous mathematical theory has not yet been established for optimal design problems with thermal radiation, which is crucial to constructing a unified theory that involves radiation as a natural extension of linear problems such as conduction and convection.

\subsection{Setting of the problem}
In this paper, for a given composite material consisting of two materials with different diffusion coefficients, we consider the material distribution that minimizes the Dirichlet energy described by the solution to the steady-state diffusion equation with the maximal monotone operator to handle various thermal transfers in a unified way.
Let $\Omega$ be a bounded domain of $\R^d$ with 
Lipschitz 
boundary $\partial \Omega$, $d\ge 1$ and
$\Omega_i\subset \Omega$ ($i=0,1$) be such that $\Omega=\Omega_0\cap\Omega_1$ and $\Omega_0\cap\Omega_1=\emptyset$.
For $\alpha,\beta>0$ such that $\beta>\alpha>0$ and $i,j\in\{1,\ldots,d\}$, the class of diffusion coefficients is defined by 
\begin{align*}
\mathcal{M}(\alpha,\beta):=\{A\in [L^{\infty}(\Omega)]^{d\times d}\colon A_{ij}=A_{ji} \text{ and } 
    \alpha|\xi|^2\le A(x)\xi\cdot\xi\le \beta|\xi|^2 \text{ for all $\xi\in \R^d$ and a.e.~$x\in \Omega$} \}.    
\end{align*}
Let $u_{\chi_{\Omega_1}}\in K$ be a weak solution to 
\begin{align}\label{eq:gE}
\begin{cases}
-\dv(A_{\chi_{\Omega_1}} \nabla u_{\chi_{\Omega_1}})=f\quad &\text{ in } \Omega,\\
-A_{\chi_{\Omega_1}} \nabla u_{\chi_{\Omega_1}}\cdot \nu
\in \boldsymbol{\beta}(u_{\chi_{\Omega_1}})
\quad &\text{ on } \partial \Omega,
\end{cases}
\end{align}
where $\nu$ is the unit outward normal vector on $\partial \Omega$, $f\in(H^1(\Omega))^\ast $, $A_{\chi_{\Omega_1}}\in \mathcal{M}(\alpha,\beta)$ is the matrix field depending on the characteristic function $\chi_{\Omega_1}\in L^{\infty}(\Omega;\{0,1\})$ given as
$$
\chi_{\Omega_1}(x)=
    \begin{cases}
        1,\quad &x\in\Omega_1, \\
        0,\quad &x\in \Omega_0,  
    \end{cases}
$$
$K$ is some closed convex subset of a Hilbert space (explained later), and $\boldsymbol{\beta}$ is a (possibly multi-valued) maximal monotone graph in $\R\times\R$ such that $\boldsymbol{\beta}(0)\ni 0$. Thus $\boldsymbol{\beta}$ satisfies 
\begin{itemize}
\item[{\rm(i)}] {\rm(Monotonicity)}
Let $G(\boldsymbol{\mathcal{\beta}})$ be the graph of $\boldsymbol{\mathcal{\beta}}$, i.e.,~
$G(\boldsymbol{\mathcal{\beta}})=\{[z,\xi]\in \R\times \R\colon\xi\in \boldsymbol{\mathcal{\beta}}(z)\}$. Then it holds that 
\[
 \langle z-w,\xi-\zeta\rangle\ge0\quad \text{ for all }\ [z,\xi],\ [w,\zeta]\in G(\boldsymbol{\mathcal{\beta}}).
\]    
\item [{\rm(ii)}] {\rm(Maximality)}
Any monotone graph
$F:\R\to 2^{\R}$
whose graph $G(F)$ involves $G(\boldsymbol{\mathcal{\beta}})$ coincides with $\boldsymbol{\mathcal{\beta}}$.
\end{itemize}
Since \(\boldsymbol{\beta}\) is a maximal monotone graph on \(\mathbb{R}\times\mathbb{R}\), 
there exists a proper, convex and lower semicontinuous function \(j:\mathbb{R}\to (-\infty,+\infty]\) such that \(\boldsymbol{\beta}=\partial j\), where \(\partial j\) denotes a subdifferential of \(j\), i.e.,~for \(w\in\mathbb{R}\),
\begin{equation*}
    \partial j(w):=\{ \xi\in \mathbb{R}\colon j(z)-j(w)\ge \xi(z-w)\ \text{ for all } z\in \mathbb{R}\}.
\end{equation*}

Now, our target minimization problem is formulated as follows{\rm \/:}
\begin{align}\label{eq:OP}
\inf_{\chi_{\Omega_1}\in \mathcal{CD}} \left\{\mathcal{E}(\chi_{\Omega_1}):=\int_\Omega A_{\chi_{\Omega_1}}(x) \nabla u_{\chi_{\Omega_1}}(x)\cdot \nabla u_{\chi_{\Omega_1}}(x)\, \d x\right\},
\end{align}
where $\mathcal{CD}$ denotes the {\it classical design domain} given as 
$$
\mathcal{CD}=\{\chi_{\Omega_1}\in L^{\infty}(\Omega;\{0,1\})\colon \|\chi_{\Omega_1}\|_{L^1(\Omega)}= \gamma |\Omega|\},
$$
$\gamma\in (0,1)$ denotes the volume ratio, and $|\Omega|$ describes the Lebesgue measure of $\Omega\subset \R^d$.

\begin{ex}[Two-material composite with different isotropic diffusion coefficients]\label{ex:lin}
\rm
Typically, 
\begin{align}\label{eq:kappa}
\kappa[\chi_{\Omega_1}]:=\alpha(1-\chi_{\Omega_1})+\beta\chi_{\Omega_1}  
\end{align}
stands for the diffusion coefficient of a composite material consisting of $\Omega_0$ with diffusion coefficient $\alpha>0$ and $\Omega_1$ with diffusion coefficient $\beta>0$, and $\kappa[\chi_{\Omega_1}]\mathbb{I}$ belongs to $\mathcal{M}(\alpha,\beta)$. Here $\mathbb{I}$ is the identity matrix of $\R^{d\times d}$.
\end{ex}

\begin{ex}[Maximal monotone operator] \label{ex:nonlin}
\rm
By noting that every subdifferential operator is maximal monotone (see, e.g.,~\cite[Theorem 2.8]{B10}), the following thermal radiation boundary condition can be taken\/{\rm :}
    \begin{equation}\label{trbc}
    \boldsymbol{\beta}(u_{\chi_{\Omega_1}})=\boldsymbol{\sigma}|u_{\chi_{\Omega_1}}|^{d}u_{\chi_{\Omega_1}},     
    \end{equation}
    where $\boldsymbol{\sigma}>0$ stands for the Stefan--Boltzmann coefficient. 
    In particular, one can represent standard linear boundary conditions with maximal monotone graph \(\boldsymbol{\beta}\).
    Indeed, if \(\boldsymbol{\beta}(w)\equiv 0\) (resp.~\(\boldsymbol{\beta}(w)=aw\) for some \(a>0\)), then the boundary condition of \eqref{trbc} is nothing but the homogeneous Neumann boundary condition (resp.~Robin boundary conditions).
    Moreover, setting
    \begin{equation*}
        \boldsymbol{\beta}_{\rm D} (w)=
        \left\{
        \begin{aligned}
            & \mathbb{R}, &&\qquad w=0,\\
            & \emptyset, &&\qquad w\neq 0,
        \end{aligned}
        \right.
    \end{equation*}
    we can understand the boundary condition of \eqref{eq:gE} as the homogeneous Dirichlet boundary condition with \(\boldsymbol{\beta}=\boldsymbol{\beta}_{\rm D}\) (for details see \cite{B71}).
    
    Hence describing the boundary conditions using the maximal monotone graph enables us to treat conduction, convection and radiation in a unified manner.
\end{ex}

\subsection{Aims and plan of the paper}
This paper aims to prove that there exist minimizers of a true relaxation problem for \eqref{eq:OP} and
construct an approximated characteristic function concerning $\chi_{\Omega_1}\in L^{\infty}(\Omega;\{0,1\})$ that achieves \eqref{eq:OP} with \eqref{eq:kappa} and \eqref{trbc} numerically as typical examples of optimal design problems with nonlinear boundary conditions.
To this end, we shall first discuss the existence of minimizers of \eqref{eq:OP} by using the direct method. 
Thus we shall take a minimizing sequence $(\chi_{\Omega_1}^n)$ in $\mathcal{CD}$, i.e.,
\begin{align}
    \lim_{n\to +\infty}\mathcal{E}(\chi_{\Omega_1}^n)= \inf_{\chi_{\Omega_1}\in \mathcal{CD}} \mathcal{E}(\chi_{\Omega_1}).
\label{eq: minseq}
\end{align}
Since $\mathcal{CD}$ involves functions that oscillate rapidly, we shall see that the minimization problem \eqref{eq:OP} is ill-posed in general due to $\chi_{\Omega_1}^{n_k}\to \theta$ weakly-$\ast$ in $L^{\infty}(\Omega)$ but $\theta\notin \mathcal{CD}$ for some subsequence $(n_k)$ and volume fraction $\theta\in L^{\infty}(\Omega;[0,1])$.
Hence showing that optimal volume fractions exist, we shall next consider how to numerically construct the optimal volume fraction $\theta\in L^{\infty}(\Omega;[0,1])$ with few intermediate set $[0<\theta<1]:=\{x\in \Omega\colon 0<\theta(x)<1\}$.  
In this paper, we shall employ the method devised in \cite{O23} as one way to resolve the so-called \emph{grayscale problem}, and
in particular, it will reveal numerically how the linearity and nonlinearity of the boundary conditions affect the material distribution. 
Furthermore, we shall present some applications of the proposed optimization algorithm and show that optimal design is superior to some physically reasonable but intuitive structures.

This paper is composed of seven sections.
Before discussing the limit of \eqref{eq: minseq}, the next section is devoted to proving the well-posedness of \eqref{eq:gE} and a corresponding homogenization theorem. 
Section \ref{S:relax} deals with the existence theorem for minimizers of a relaxation problem with respect to the volume fraction.
In particular, we shall prove that the minimum value of the relaxation problem coincides with that of the original design problem \eqref{eq:OP} and discuss how to construct the volume fraction that achieves the minimum value numerically. In section \ref{S:LSM}, we shall consider a perimeter constraint problem via the positive part of the level set function as an approximate problem for the relaxation problem. In particular, the existence theorem for minimizers of the approximation problem will be proved, and then the energy convergence of minimizers with respect to the perturbation parameter ($\e>0$) will also be shown. Furthermore, we shall provide a numerical algorithm that does not raise the grayscale problem. 
Section 5 details the validity of the numerical calculations and the effect of the nonlinearity of boundary conditions for optimized configurations, and then numerical examples for thermal radiation problems as more physical settings will be presented in Section 6.
The final section concludes this paper. 

\section{Homogenization problem for \eqref{eq:gE}}
As a preliminary to discuss the existence of minimizers in \eqref{eq:OP}, we consider the homogenization problem for \eqref{eq:gE} with $A_{\chi_{\Omega_1}}$ being replaced by $A^n$.
Here and henceforth, we only consider the single-valued case of $\boldsymbol{\beta}$ for simplicity. To define a weak solution to \eqref{eq:gE},
we first prepare some notations.
For simplicity, we set \(V=H^1(\Omega)\), $A_{\chi_{\Omega_1}}=A$ and $u_{\chi_{\Omega_1}}=u$.
Let a closed convex subset $K\subset V$ be given by 
$K=\{v\in V\colon  j(v)\in L^1(\partial \Omega)\}$, and \(\langle\cdot,\cdot\rangle\) denotes the \(V^*\)-\(V\) dual coupling.
We define some functionals
\begin{align*}
& a(u,v):= \int_{\Omega} A(x)\nabla u(x)\cdot \nabla v(x)\, \d x &&\text{ for all }\ u,v\in V,\\
\nonumber
& J(u):= \int_{\partial \Omega} j(u)(x)\, \d\sigma \quad &&\text{ for all }\ u\in K.
\end{align*}
In addition, we assume that there exist $\delta>0$ and $C_0\in\mathbb{R}$ such that
\begin{equation}
\label{hyp:J}
J(v) \ge \delta \|v\|_{L^2(\partial \Omega)}^2 - C_0 \qquad \text{ for all }\ v\in K.
\end{equation}
The typical example of $j:\R\to (-\infty,\infty]$ satisfying \eqref{hyp:J} is $j(w)=\tfrac{1}{r}|w|^{r}$ (i.e.,~$\boldsymbol{\beta}(w)=|w|^{r-2}w$) for $r>1$.
We first touch on an inequality, which plays an important role in dealing with nonlinear boundary conditions.
\begin{lem}
\label{poincare}
    Let \(\Omega\subset \mathbb{R}^d\) be a bounded Lipschitz domain. Then
    there exists a constant \(C>0\) such that 
    \begin{equation}
        \|v\|_{L^2(\Omega)}^2 \le C \left( \|\nabla v \|_{L^2(\Omega)}^2 + \|v\|_{L^2(\partial\Omega)}^2 \right)
        \label{eq:pineq}
    \end{equation}
    for all \(v\in H^1(\Omega)\). As a consequence, \((\|\nabla v \|_{L^2(\Omega)}^2 + \|v\|_{L^2(\partial\Omega)}^2)^{1/2}\) is a equivalent norm in \(H^1(\Omega)\).
\end{lem}
\begin{proof}
    Suppose that for any \(n\in\mathbb{N}\) there exist \((v_n)\subset H^1(\Omega)\) such that 
    \begin{equation*}
        \|v_n\|_{L^2(\Omega)}^2 > n \left( \|\nabla v_n \|_{L^2(\Omega)}^2 + \|v_n\|_{L^2(\partial\Omega)}^2 \right).
    \end{equation*}
    Setting \(w_n=v_n\|v_n\|_{L^2(\Omega)}^{-1}\), we have \(\|w_n\|_{L^2(\Omega)}=1\) and 
    \begin{equation}\label{poin}
        \|\nabla w_n\|_{L^2(\Omega)}^2 + \|w_n\|_{L^2(\partial\Omega)}^2 < \frac{1}{n}.
    \end{equation}
    Hence, since \(\Omega\) is bounded and \((w_n)\) is bounded in \(H^1(\Omega)\), we can deduce that there exists a subsequence of \((w_n)\) (denoted by \((w_n)\) again) and \(w\in H^1(\Omega)\) such that 
    \begin{align*}
        &w_n \to w &&\mbox{weakly in } H^1(\Omega),\\
        &w_n \to w &&\mbox{strongly in } L^2(\Omega) \mbox{ and } L^2(\partial\Omega).
    \end{align*}
    From \eqref{poin}, we can see that \(w=0\) a.e. on \(\partial\Omega\).
    Moreover, the lower semicontinuity leads \(\|\nabla w\|_{L^2(\Omega)}\le 
    \liminf_{n\to +\infty} \|\nabla w_n\|_{L^2(\Omega)} = 0\), which implies \(w\) is constant, especially \(w=0\).
    However, this contradicts that \(\|w\|_{L^2(\Omega)}=1\).
    The equivalence of the norm follows from \eqref{eq:pineq} and the trace theorem. This completes the proof.
\end{proof}

In the framework of the variational inequality,
the weak solution of \eqref{eq:gE} is defined as follows\/{\rm :}

\begin{defi}[Weak solution of \eqref{eq:gE}]\label{D:var_ineq}
For given $f\in V^\ast$, a function \(u\in K\) is said to be a weak solution of \eqref{eq:gE} if and only if the following inequality holds\/{\rm :}
\begin{equation}
\label{var_ineq}
a(u,v-u) + J(v) - J(u) \ge \left\langle f,v-u \right\rangle \quad \text{ for all }\ v\in K.
\end{equation}
\end{defi}

\begin{rmk}[Weak form of \eqref{eq:gE}]\label{R:weak_form}
    \rm
    We note that the definition of weak solutions for \eqref{eq:gE} is defined by employing the variational inequality \eqref{var_ineq}.
    From physical motivation, we deal exclusively with the case of \(\boldsymbol{\beta}(u)=\boldsymbol{\sigma} |u|^{r-2}u\) (\(r>1\), \(\boldsymbol{\sigma}>0\)). 
In this case, setting $v=u\pm\lambda \varphi\in K$ for all $\varphi\in K$ and $\lambda>0$ and letting $\lambda\to 0_+$, we see by $\langle J'(u), \phi\rangle_{L^{r}(\partial\Omega)}= \langle \boldsymbol{\sigma}|u|^{r-2}u, \phi\rangle_{L^{r}(\partial\Omega)}$ that 
 \begin{equation}
    \label{weak_form}
    a(u,\varphi) + \boldsymbol{\sigma}\int_{\partial\Omega} |u|^{r-2}u(x) \varphi(x)\, \d\sigma = 
    \langle f, \varphi\rangle  \qquad\mbox{for all } \varphi \in K.
    \end{equation}
Thus the weak solution $u\in K$ to \eqref{var_ineq} satisfies \eqref{weak_form}. In particular, by the monotonicity of $\boldsymbol{\beta}$, uniqueness of solutions to \eqref{weak_form} also follows, and then 
 every weak solution of \eqref{var_ineq} coincides with the solution of this weak form \eqref{weak_form}.
The reason why weak solutions are defined by \eqref{var_ineq} instead of \eqref{weak_form} will be explained in Remark \ref{R:UBI}. 
\end{rmk}

\subsection{Solvability of \eqref{eq:gE}} 
Let us briefly touch on the existence and uniqueness of a weak solution to \eqref{eq:gE}.

\begin{thm}[Existence and uniqueness of weak solutions to \eqref{eq:gE}]\label{T:well-posedness}
There exists a unique weak solution to \eqref{eq:gE}.
\end{thm}

\begin{proof}
We first seek a function satisfying the following minimizing problem\/{\rm :} 
\begin{equation}
\text{For given $f\in V^\ast$, find $u\in K$ such that }
    \label{eq:mini}
    E(u) = \inf_{v\in K} E(v),
\end{equation}
where $E\colon V\supset K\to (-\infty,\infty]$ is the functional defined by
\begin{equation}\label{eq:defE}
    E(v) = \frac{1}{2} a(v,v) - \langle f,v \rangle + J(v).
\end{equation}
It is easy to see that \(E\) is convex, lower semicontinuous and \(E\not\equiv\infty\).
Moreover, it follows from \eqref{hyp:J} and Lemma \ref{poincare} that
\begin{equation*}
    \lim_{v\in K,\ \|v\|_{V}\to +\infty} E(v) = +\infty,
\end{equation*} 
which along with \cite[Corollary 3.23]{B11} ensures the existence of
$u\in K$ satisfying \eqref{eq:mini}.
Thus setting $E_1$ and $E_2$ as the first two terms and the third term of the right-hand side in \eqref{eq:defE}, respectively, we have $E_1(u)+E_2(u)=\inf_{v\in K}(E_1+E_2)$, and hence, \cite[Theorem 1.6]{Lions} yields
$$
\langle E_1'(u),v-u\rangle+E_2(v)-E_2(u)\ge 0 \quad \text{ for all }\ v\in K,
$$ 
which implies \eqref{var_ineq}. 
The uniqueness of weak solutions follows immediately from the strict convexity of \(E\).
Indeed, let \(u_1\) and \(u_2\in K\) be two minimizers for \eqref{eq:mini}.
If \(u_1\neq u_2\), then
\begin{align*}
    E(u_i) \le E\left(\frac{u_1+u_2}{2}\right) < \frac{1}{2} \left\{ E(u_1) + E(u_2) \right\} \qquad (i=1,2),
\end{align*}
that is, \(E(u_1)<E(u_2)\) and \(E(u_2)<E(u_1)\).
This completes the proof.
\end{proof}

\subsection{Homogenization problem for \eqref{eq:gE}}

As for the proof of the homogenization theorem for \eqref{eq:gE} with $A_{\chi_{\Omega_1}}$ being replaced by $A^n$, the following notion developed by Murat-Tartar is very useful.  

\begin{defi}[H-convergence, cf.~\cite{Murat-Tartar}]
Let $A\in \mathcal{M}(\alpha,\beta)$.  
For $n>0$, a sequence $A^n\in \mathcal{M}(\alpha,\beta)$ H-converges to an element $A_{\rm hom}\in [L^{\infty}(\Omega)]^{d\times d}$ {\rm(}denoted by $A^n\overset{H}{\to} A_{\rm hom}${\rm)} if and only if, for any $\omega\Subset \Omega$ and any $f \in H^{-1}(\omega)$, the weak solution $u^n\in H^1_0(\omega)$ of
$$
-\dv(A^n \nabla u^n)=f\quad \text{ in } H^{-1}(\omega)
$$
is such that 
\begin{align*}
    u^n&\to u_{\rm hom} &&\text{ weakly in }  H^1_0(\omega),\\
   A^n\nabla u^n&\to A_{\rm hom}\nabla u_{\rm hom} &&\text{ weakly in }  [L^2(\omega)]^d, 
\end{align*}
where \(u_{\rm hom}\in H^1_0(\omega)\) is the weak solution to
$$
-\dv(A_{\rm hom} \nabla u_{\rm hom})=f\quad \text{ in } H^{-1}(\omega).
$$
\end{defi}

Then we have the following 
\begin{thm}[Homogenization theorem]\label{T:H-conv}
Let $u^n\in K$ be the unique weak solution to
\begin{align}\label{eq:gE1}
\begin{cases}
-\dv(A^n \nabla u^n)=f\quad &\text{ in } \Omega,\\
-A^n \nabla u^n\cdot \nu
= \boldsymbol{\beta}(u^n)
\quad &\text{ on } \partial \Omega,
\end{cases}
\end{align}
where $A^n\in \mathcal{M}(\alpha,\beta)$, $f\in V^*$ and $\boldsymbol{\beta} $ is a maximal monotone operator in $\R$.
Then there exist a {\rm(}not relabeled{\rm)} subsequence of $(n)$, $u_{\rm hom}\in K$, the homogenized matrix $A_{\rm hom}\in \mathcal{M}(\alpha',\beta')$ and $\beta'>\alpha'>0$  such that  
\begin{align*}
    u^n&\to u_{\rm hom} &&\text{ weakly in }  V,\\
   A^n\nabla u^n&\to A_{\rm hom}\nabla u_{\rm hom} &&\text{ weakly in }  [L^2(\Omega)]^d.
\end{align*}
Moreover, $u_{\rm hom}\in K$ is a weak solution to the homogenized equation,
\begin{align}\label{eq:hom}
\begin{cases}
-\dv(A_{\rm hom} \nabla u_{\rm hom})=f\quad &\text{ in } \Omega,\\
-A_{\rm hom} \nabla u_{\rm hom}\cdot \nu
= \boldsymbol{\beta}(u_{\rm hom})
\quad &\text{ on } \partial \Omega.
\end{cases}
\end{align}
\end{thm}

\begin{proof}
By \eqref{var_ineq} and \eqref{hyp:J}, it holds that
\begin{equation*}
a^n(u^n,u^n) + \delta\|u^n\|_{L^2(\partial \Omega)}^2 \le C_0 + \langle f,u^n \rangle,
\end{equation*}
which along with Lemma \ref{poincare} yields \(\|u^n\|_{V}\le C\). 
By virtue of the boundedness of \((u^n)_{n>0}\) in \(K\), there exist a (not relabeled) subsequence of $(n)$ and \(u_{\rm hom}\in V\) such that
\begin{align}
	\label{conver:H1}
	&u^n \to u_{\rm hom} \qquad\mbox{weakly in } V,\\[2mm]
	\nonumber
	&u^n \to u_{\rm hom} \qquad\mbox{strongly in } L^2(\Omega)\cap L^2(\partial \Omega).
\end{align}
Furthermore, by the $H$-compactness \cite[Theorem 2]{Murat-Tartar}, $A^n\overset{H}{\to} A_{\rm hom}$ also holds.
In particular, since it is obvious that \(u^n\pm\varphi\in K\) for any \(\varphi\in H^1_0(\Omega)\), by choosing \(v=u^n\pm\varphi\) in \eqref{var_ineq}, it follows that
\begin{equation*}
a^n(u^n,\varphi) = \langle f,\varphi \rangle \qquad\text{ for all } \varphi\in H^1_0(\Omega),
\end{equation*}
that is, 
\begin{equation}
	\label{eq:homo1}
-\mathrm{div}\left( A^n \nabla u^n \right) = f \qquad\mbox{in } H^{-1}(\Omega).
\end{equation}
Hence applying \cite[Theorem 1]{Murat-Tartar} to \eqref{eq:homo1}, we obtain
\begin{equation}
	\label{conver:flux}
	A^n\nabla u^n \to A_{\rm hom}\nabla u_{\rm hom} \qquad\mbox{weakly in } [L^2(\Omega)]^d,
\end{equation} 
\begin{equation}
	\label{conver:dis}
	\int_{\Omega} A^n(x)\nabla u^n(x)\cdot\nabla u^n(x) \varphi(x)\, \d x \to \int_{\Omega} A_{\rm hom}(x)\nabla u_{\rm hom}(x)\cdot\nabla u_{\rm hom}(x) \varphi(x) \, \d x 
\end{equation}
for all $\varphi\in C_{\rm c}^{\infty}(\Omega)$. 
Choosing \(\varphi\in C_{\rm c}^\infty(\Omega)\) such that \(0 \le \varphi \le 1\) in \eqref{conver:dis}, 
we can deduce that
\begin{equation*}
	\int_{\Omega} A_{\rm hom}(x)\nabla u_{\rm hom}(x)\cdot\nabla u_{\rm hom}(x) \varphi(x)\, \d x 
	\le \liminf_{n\to +\infty} a^n(u^n,u^n) \ \text{ for all }\ \varphi\in C_{\rm c}^{\infty}(\Omega),
\end{equation*}
whence follows 
the lower bound inequality,
\begin{equation}
	\label{liminf:bilinear}
	a_{\rm hom}(u_{\rm hom},u_{\rm hom}) 
 =
 \sup_{\substack{\varphi\in C_{\rm c}^{\infty}(\Omega),\\ \|\varphi\|_{L^{\infty}(\Omega)}\le 1}}
 \int_{\Omega}A_{\rm hom}(x)\nabla u_{\rm hom}(x)\cdot \nabla u_{\rm hom}(x)\varphi(x)\, \d x
\le
 \liminf_{n\to +\infty} a^n(u^n,u^n)
\end{equation}
(see \cite[4.26]{B11}). Here we 
note that the upper bound inequality is more delicate (see Remark \ref{R:UBI} below).
Since \(K\) is closed convex subset in \(V\), \(u_{\rm hom}\in K\) holds.
Therefore the weakly lower semicontinuous on \(L^2(\partial \Omega)\) of \(J\) ensures that
\begin{equation*}
	J(u_{\rm hom}) \le \liminf_{n\to +\infty} J(u^n),
\end{equation*}
which together with the definition of weak solutions, \eqref{conver:H1}, \eqref{conver:flux} and \eqref{liminf:bilinear} yields
\begin{equation*}
	a_{\rm hom}(u_{\rm hom},v-u_{\rm hom}) + J(v) - J(u_{\rm hom}) \ge \left\langle f,v-u_{\rm hom} \right\rangle \quad\text{ for all }\ v\in K.
\end{equation*}
Thus \(u_{\rm hom}\in K\) turns out to be a weak solution to the homogenized equation.
This completes the proof.
\end{proof}

\begin{rmk}[Energy convergence]\label{R:UBI}
\rm
Under the usual definition via the weak form, the following weak convergence will be required in the proof of Theorem \ref{T:H-conv}\/{\rm :} 
$$
\boldsymbol{\beta}(u^n)\to \boldsymbol{\beta}(u_{\rm hom})
\quad \text{ weakly in } L^{\ell}(\partial \Omega)
$$
for some $1<\ell<+\infty$ (e.g., $\ell=r/(r-1)$ with $\boldsymbol{\beta}(w)=|w|^{r-2}w$).
However, this proof is more complicated in general.
Thus Definition \ref{D:var_ineq} is more reasonable than the usual definition using the weak form \eqref{weak_form}.
Moreover, we note that even if the homogeneous Dirichlet boundary condition, the upper bound inequality,
\begin{equation}
\limsup_{n\to +\infty} a^n(u^n,u^n)
\le
a_{\rm hom}(u_{\rm hom},u_{\rm hom}) 
\label{eq:ubi}
\end{equation}
is derived with the aid of the weak form for the homogenized equation. 
On the other hand, as soon as $u_{\rm hom}$ turns out to be a weak solution to the homogenized equation satisfying the weak form
(for instance, 
it suffices to assume that 
$$
\lim_{\lambda\to 0_+}\frac{J(w+\lambda v)-J(w)}{\lambda}=\int_{\partial\Omega}\partial j(w)(x)v(x)\, \d \sigma\quad \text{ for all } v\in K),
$$
we readily obtain \eqref{eq:ubi} by noting that 
\begin{align*}
\limsup_{n\to +\infty}   a^n(u^n,u^n)
&=
\lim_{n\to +\infty}\langle f,u^n\rangle- \liminf_{n\to +\infty} J(u^n)\\
&\le 
\langle f,u_{\rm hom}\rangle - J(u_{\rm hom})
=  a_{\rm hom}(u_{\rm hom},u_{\rm hom}).
 \end{align*}
Therefore convergence of the energy $\mathcal{E}(\chi_{\Omega_1}^n)$ in \eqref{eq: minseq} can be obtained, which is applied to the proof of Theorem \ref{T:existence} below.  
\end{rmk}

\begin{rmk}[Qualitative properties of homogenized matrices]
\rm
It is noteworthy that the homogenized matrix $A_{\rm hom}\in [L^{\infty}(\Omega)]^{d\times d}$ is completely unaffected by (nonlinear) boundary conditions, and then it is also characterized exactly as in the case of the homogeneous Dirichlet boundary condition. 
Hence, in case $d=1$, it can be written as the so-called harmonic mean. Conversely (i.e.,~$d\ge 2$), it cannot be written explicitly in general. 
However, since $A$ is symmetric, the following upper and lower bounds can be obtained\/{\rm :}
\begin{align}
    \underline{A}\xi\cdot \xi\le A_{\rm hom}\xi\cdot \xi\le \overline{A}\xi\cdot \xi\quad \text{ for all } \xi\in \R^d,   \label{eq:est-hom}
\end{align}
where $\underline{A}$ is the inverse of the weak limit of $A^{-1}$ and $\overline{A}$ is the weak limit of $A$ (see, e.g.,~\cite[Theorem 1.3.14]{A02}).
\end{rmk}

\section{Relaxation problem for \eqref{eq:OP}}\label{S:relax}

This section is devoted to proving the existence theorem for minimizers of a relaxation problem (see Theorem \ref{T:existence} below). Thanks to Theorem \ref{T:H-conv}, most of the proof relies on the results in \cite[Theorem 3.2.1]{A02}; however, we shall show it for completeness and the reader's convenience. 
Furthermore, we shall describe how to construct a candidate for the minimizers of the relaxation problem numerically. 
In what follows, we set $A_{\chi_{\Omega_1}}=\kappa[\chi_{\Omega_1}]$ as in \eqref{eq:kappa} and write $\kappa[\chi_{\Omega_1}]=\kappa$ and $\chi_{\Omega_1}=\chi$ for simplicity.

\subsection{Existence theorem for minimizers}
Since the limit of the minimizing sequence does not belong to $\mathcal{CD}$ in general, we consider the following relaxation problem\/{\rm :} 
\begin{align}\label{eq:RP}
\inf_{(\theta, \kappa_{\rm hom})\in \mathcal{RD}} \mathcal{E}_{\rm hom}(\theta, \kappa_{\rm hom}),
\end{align}
where $\mathcal{RD}$ is the {\it relaxed design domain} defined by 
\begin{align*}
\mathcal{RD}
:= \left\{ 
\begin{aligned}
(\theta, \kappa_{\rm hom})\in L^{\infty}(\Omega;[0,1]\times \R^{d\times d})\colon
&
\text{there exists }
(\chi^n,\kappa^n)\in  L^{\infty}(\Omega;\{0,1\}\times \{\alpha,\beta\})\\ 
&\text{such that }
\kappa^n=\alpha(1-\chi^n)+\beta\chi^n,\\
&\chi^n\to \theta \text{ weakly-$\ast$ in } L^{\infty}(\Omega;[0,1]),\\
&\kappa^n \mathbb{I} \overset{H}{\to} \kappa_{\rm hom} \text{ and } \|\theta\|_{L^1(\Omega)}= \gamma |\Omega|
\end{aligned}
\right\},
\end{align*}
$\mathcal{E}_{\rm hom}:L^{\infty}(\Omega;[0,1]\times \R^{d\times d})\to \R$ is the relaxed Dirichlet energy given as 
$$
\mathcal{E}_{\rm hom}(\theta,\kappa_{\rm hom})
=\int_{\Omega}\kappa_{\rm hom}(x)\nabla u_{\rm hom}(x)\cdot \nabla u_{\rm hom}(x)\, {\rm d}x
$$
and $u_{\rm hom}\in K$ is a weak solution to the homogenized equation \eqref{eq:hom} with $A_{\rm hom}=\kappa_{\rm hom}$.

Then we see that the relaxation problem \eqref{eq:RP} is a true relaxation of the original design problem \eqref{eq:OP} in the following sense\/{\rm :}
\begin{thm}[Existence theorem for minimizers of \eqref{eq:RP}]\label{T:existence}
Let $u^n\in K$ be a weak solution to \eqref{eq:gE1} satisfying the weak from.
Let $u_{\rm hom}\in K$ be a weak solution to the homogenized equation
\eqref{eq:hom}. 
Then there exists at least one minimizer of \eqref{eq:RP}.
Furthermore, it holds that
\begin{equation}\label{eq:mini-charac}
\inf_{\chi\in \mathcal{CD}} \mathcal{E}(\chi)
=
\min_{(\theta,\kappa_{\rm hom})\in \mathcal{RD}} \mathcal{E}_{\rm hom}(\theta, \kappa_{\rm hom}),
\end{equation}
and every minimizer of \eqref{eq:RP} is characterized as a limit of the minimizing sequence in \eqref{eq:OP}, that is, for any minimizing sequence $(\chi^n)$ in $\mathcal{CD}$, there exist a {\rm(}not relabeled{\,\rm)} subsequence of $(n)$ and $(\theta^\ast,\kappa_{\rm hom}^\ast)\in L^{\infty}(\Omega;[0,1]\times \R^{d\times d})$ such that
\begin{equation}
    \chi^n\to \theta^\ast \text{ weakly-$\ast$ in } L^{\infty}(\Omega;[0,1]),\quad 
\kappa^n\mathbb{I}
\overset{H}{\to} \kappa_{\rm hom}^\ast 
\label{eq:cnv}
\end{equation}
and
\begin{equation}\label{eq:mini-charac2}
\inf_{\chi \in \mathcal{CD}} \mathcal{E}(\chi)
=
\mathcal{E}_{\rm hom}(\theta^\ast,\kappa_{\rm hom}^\ast).
\end{equation}
Conversely, every minimizer in \eqref{eq:RP} is attained by a limit of the minimizing sequence $(\chi^n)$ in  \eqref{eq:OP}.
\end{thm}

\begin{proof}
Let $(\chi^n)$ be a minimizing sequence in $\mathcal{CD}$ for \eqref{eq:OP}.
Due to $\|\chi^n\|_{L^{\infty}(\Omega)}\le 1$ and Theorem \ref{T:H-conv}, there exist
a (not relabeled) subsequence of $(n)$ and $(\theta^\ast, \kappa_{\rm hom}^\ast)\in L^{\infty}(\Omega;[0,1]\times \R^{d\times d})$ such that \eqref{eq:cnv}
and 
$$f
\int_\Omega \theta^\ast(x)\, \d x=\lim_{n\to +\infty}\int_\Omega \chi^n(x)\, \d x=\gamma|\Omega|.
$$
Moreover, one can derive by Remark \ref{R:UBI} that
\begin{align}
    \inf_{\chi\in \mathcal{CD}}\mathcal{E}(\chi)
    =\lim_{n\to +\infty} \mathcal{E}(\chi^n)&=\mathcal{E}_{\rm hom} (\theta^\ast,\kappa_{\rm hom}^\ast). \label{eq:minimizer}
\end{align}
In particular, the above continuity \eqref{eq:minimizer} is valid for non-minimizing sequences.

Now, we show that $(\theta^\ast,\kappa_{\rm hom}^\ast)\in \mathcal{RD}$ is a minimizer of \eqref{eq:RP}. For any $(\theta,\kappa_{\rm hom})\in\mathcal{RD}$, there exists $(\chi^n,\kappa^n)\in L^{\infty}(\Omega;\{0,1\}\times\{\alpha,\beta\})$ such that
\begin{align}\label{eq:RD}
    \chi^n\to \theta\quad \text{ weakly-$\ast$ in } L^{\infty}(\Omega;[0,1])\quad \text{ and }\quad 
    \kappa^n \mathbb{I}\overset{H}{\to} \kappa_{\rm hom}.
\end{align}
In particular, we can construct the sequence $(\chi^n)$ in $L^{\infty}(\Omega;\{0,1\})$ such that $\|\chi^n\|_{L^1(\Omega)}=\gamma|\Omega|$, i.e.,~$\chi^n\in \mathcal{CD}$. 
Indeed, 
let $\hat{\chi}^{n}\in L^{\infty}(\Omega;\{0,1\})$ be such that $\|\hat{\chi}^{n}\|_{L^1(\Omega)}=\gamma|\Omega|$.
Combining $\|\hat\chi^n\|_{L^{\infty}(\Omega)}\le 1$ with
$$
\lim_{n\to \infty}
\int_{\Omega} \hat\chi^n(x)\, \d x
=
\gamma|\Omega|=
\int_{\Omega}\theta(x)\, \d x,
$$
we have
$$
\hat\chi^n\to \theta\quad \text{ weakly-$\ast$ in }\ L^{\infty}(\Omega). 
$$
Then defining $\Omega^n\subset \Omega$ by 
$$
\Omega^{n}=\{ x\in \Omega\colon \chi^n(x)\neq \hat{\chi}^n(x) \text{ for a.e.~} x\in\Omega\},
$$
one obtain $|\Omega^n|\to 0$ as $n\to +\infty$, which along with   
the locality of the $H$-convergence (see  \cite[(ii) of Proposition 1]{Murat-Tartar}) yields $\kappa[\hat\chi^n]\mathbb{I}\overset{H}{\to} \kappa_{\rm hom}$, and therefore, $(\tilde\chi^n)$ turns out to be the desired sequence.
Hence the continuity of $\mathcal{E}$ and \eqref{eq:OP} ensure that
\begin{align*}
    \mathcal{E}_{\rm hom}(\theta,\kappa_{\rm hom})=\lim_{n\to +\infty}\mathcal{E}(\chi^n)\ge \inf_{\chi\in\mathcal{CD}}\mathcal{E}(\chi),
\end{align*}
which along with \eqref{eq:minimizer} yields \eqref{eq:mini-charac} and \eqref{eq:mini-charac2}.  

On the other hand, let $(\theta,\kappa_{\rm hom})\in\mathcal{RD}$ be a minimizer of \eqref{eq:RP}.
As already mentioned the above, \eqref{eq:RD} follows for some $\chi^n\in\mathcal{CD}$ and some $\kappa^n\in L^{\infty}(\Omega;\{\alpha,\beta\})$ such that $\kappa^n=\alpha(1-\chi^n)+\beta\chi^n$, and moreover,
$\mathcal{E}_{\rm hom}(\theta,\kappa_{\rm hom})=\lim_{n\to +\infty}\mathcal{E}(\chi^n)$ also holds, which implies that $(\chi^n)$ is a minimizing sequence in \eqref{eq:OP}. 
This completes the proof. 
\end{proof}

\begin{rmk}[Interpretation of Theorem \ref{T:existence}]\label{R:opt-d}
\rm We note that the following facts\/{\rm :}
\begin{itemize}
\item[(i)]
Theorem \ref{T:existence} asserts that at least one minimizer exists in the relaxation problem \eqref{eq:RP}, which gives the same minimum value as the original design problem \eqref{eq:OP}. However, there is no guarantee for the uniqueness of minimizers.
In particular, if we further add geometric constraints (see, e.g.,~\cite{AB93} for perimeter constraints) such that $\chi^n$ converges some characteristic function a.e.~in $\Omega$, the above proof ensures the existence of minimizers of the original design problem \eqref{eq:OP} (with geometric constraints). 

\item[(ii)]
As for the relaxation problem \eqref{eq:RP}, the volume fraction $\theta\in L^{\infty}(\Omega;[0,1])$ takes a value other than $\{0,1\}$ a.e.~in $\Omega$.
Thus there are intermediate sets that are neither the material with diffusion coefficient $\alpha>0$ (i.e.,~$\Omega_0\subset \Omega$) nor the material with diffusion coefficient $\beta>0$ (i.e.,~$\Omega_1\subset \Omega$), and optimal volume fractions are characterized by using intermediate sets.
In terms of the original design problem \eqref{eq:OP}, it is necessary to construct $(\theta,\kappa_{\rm hom})\in \mathcal{RD}$ that attains the value close to the minimum in \eqref{eq:OP} such that the so-called gray-scale problem is rarely raised. 
\end{itemize}
\end{rmk}

\subsection{Numerical algorithm for optimization of volume fractions}
In this section, we describe a method to construct candidates for optimal volume fractions in \eqref{eq:RP} numerically such that the minimum value of \eqref{eq:OP} is achieved.
In the rest of this paper, 
we adapt \(\|\cdot\|_{V}=(\|\nabla\cdot\|_{L^2(\Omega)}^2 + \|\cdot\|_{L^2(\partial \Omega)}^2)^{1/2}\) which is a norm in \(V=H^1(\Omega)\) equivalent to the usual one, and then
we set \(\boldsymbol{\beta}(w)=\boldsymbol{\sigma}|w|^{r-2}w\) (i.e.,~\(j(w)=\frac{\boldsymbol{\sigma}}{r}|w|^r\)), \(r\ge 2\) and write
$\kappa_{\rm hom}=\kappa$, $u_{\rm hom}=u$ and  
$\mathcal{E}_{\rm hom}(\theta,\kappa_{\rm hom})=\mathcal{E}_{\rm hom}(\kappa)$ for simplicity.
In addition, we assume that $f\in L^2(\Omega;\R_+)$ to get the following 
\begin{lem}[Nonnegativity of $u$]\label{L:nonnegative}
Let $u\in K$ be a unique weak solution to the {\rm(}homogenized{\rm)} state equation,
\begin{align}\label{eq:rE}
\begin{cases}
-\dv(\kappa \nabla u)=f\quad &\text{ in } \Omega,\\
-\kappa \nabla u\cdot \nu=\boldsymbol{\sigma}|u|^{r-2}u \quad &\text{ on } \partial \Omega.
\end{cases}
\end{align}
Suppose that \(f\ge 0\), \(\not\equiv 0\). Then it holds that $u\ge 0$ a.e.~in $\Omega$.
\end{lem}

\begin{proof}
Multiplying \(u^-:=\max(-u,0)\) by \eqref{eq:rE} and using integration by parts, we can derive that 
\begin{align*}
    0 \le \int_{\Omega} f(x)u^-(x)\, \d x 
    &= \int_{\Omega}\kappa(x) \nabla u(x) \cdot \nabla u^-(x)\, \d x + \boldsymbol{\sigma} \int_{\partial \Omega} |u|^{r-2} u(x) u^-(x)\, \d\sigma \displaybreak[0]\\[2mm]
    &=\int_{\{u\le 0\}}\kappa(x) \nabla u(x) \cdot \nabla (-u)(x)\, \d x + \boldsymbol{\sigma} \int_{u\le 0} |u|^{r-2} u(x) (-u)(x)\,  \d\sigma\displaybreak[0]\\[2mm]
    &= - \int_{\{u\le 0\}}\kappa(x) \nabla (-u)(x) \cdot \nabla (-u)(x) \,\d x - \boldsymbol{\sigma} \int_{u\le 0} |u(x)|^{r-2} (-u(x))^2\, \d\sigma\displaybreak[0]\\[2mm]
    &\le - \alpha' \| \nabla u^-\|_{L^2(\Omega)}^2 - \boldsymbol{\sigma} \int_{\partial \Omega} |u(x)|^{r-2}|u^-(x)|^2\, \d\sigma,
\end{align*}
which implies
$    \nabla u^- =  u^-|_{\partial \Omega} = 0$,
and therefore, \(u\ge 0\) a.e.~in \(\Omega\). This completes the proof.
\end{proof}

Thanks to Lemma \ref{L:nonnegative}, $|u|^{r-2}u$ can be described as $u^{r-1}$ below. 
We first derive the Fr\'echet derivative of $\mathcal{E}_{\rm hom}:\mathcal{M}(\alpha',\beta')\to\R$.
\begin{prop}[Sensitivity analysis for \eqref{eq:RP}]\label{P:d-sensitivity}
 Let $u\in K$ be a nonnegative weak solution to \eqref{eq:rE} and let $v\in K$ be a weak solution to the {\rm (}homogenized{\rm )} adjoint equation,
\begin{align}\label{eq:adE}
\begin{cases}
-\dv(\kappa \nabla v)=f\quad &\text{ in } \Omega,\\
-\kappa \nabla v\cdot \nu=\boldsymbol{\sigma}((r-1)u^{r-2}v+ru^{r-1}) \quad &\text{ on } \partial \Omega.
\end{cases}
\end{align}
Then $\mathcal{E}_{\rm hom}$ is differentiable at $\kappa$, and it holds that
\begin{equation}
\left\langle \mathcal{E}_{\rm hom}'(\kappa),h\right\rangle_{[L^{\infty}(\Omega)]^{d\times d}}
=
-\int_{\Omega}h(x)\nabla u(x)\cdot \nabla v(x)\, \d x
    \label{eq:d-sensitivity}
\end{equation}
for any $h\in [L^{\infty}(\Omega)]^{d\times d}$. 
\end{prop}

\begin{proof}
Define $\mathcal{L}:\mathcal{M}(\alpha',\beta')\times K\times K \to\R$ by
\begin{align*}
\lefteqn{\mathcal{L}(\kappa,u,w_\kappa)}\nonumber\\
&=
\int_{\Omega}\kappa(x)\nabla u(x)\cdot \nabla w_\kappa(x)\, \d x
-
\int_{\partial \Omega}\boldsymbol{\sigma}u^{r-1}(x)(u(x)-w_\kappa(x))\, \d \sigma
+
\int_\Omega f(x)(u(x)-w_\kappa(x))\, \d x, 
\end{align*}
where $\mathcal{M}(\alpha',\beta')\ni \kappa\mapsto w_\kappa\in K$ is a differentiable at $\kappa$. 
Note that, for any $h\in [L^{\infty}(\Omega)]^{d\times d}$,
\begin{align*}
\left\langle \mathcal{L}'(\kappa,u, w_\kappa),h\right\rangle_{[L^{\infty}(\Omega)]^{d\times d}}
=
\left\langle \partial_{\kappa}\mathcal{L}(\kappa,u,w_\kappa), h\right\rangle_{[L^{\infty}(\Omega)]^{d\times d}}
+
\left\langle \partial_{u}\mathcal{L}(\kappa,u,w_\kappa),u'h \right\rangle_V
+
\left\langle \partial_{w_\kappa}\mathcal{L}(\kappa,u,w_\kappa),w'_\kappa h \right\rangle_V.
\end{align*}
Here we used the fact that $\kappa\mapsto u=u_\kappa$ is differentiable (see Lemma \ref{L:diffu} below). 
We derive by the symmetry of $\kappa$ that, for any $\varphi\in K$,
\begin{align*}
\langle\partial_{u}\mathcal{L}(\kappa,u,w_\kappa),\varphi\rangle_{V}
&=
\int_{\Omega} \kappa(x)\nabla w_\kappa(x)\cdot \nabla \varphi(x)\, \d x\\
&\quad -\int_{\partial \Omega}\boldsymbol{\sigma}[ru^{r-1}(x)
-(r-1)u^{r-2}(x)w_\kappa(x)]\varphi(x)\,\d \sigma
+
\int_\Omega f(x)\varphi(x)\, \d x,
\end{align*}
whence follows 
$
\langle \mathcal{E}_{\rm hom}'(\kappa),h\rangle_{[L^{\infty}(\Omega)]^{d\times d}}
=
\langle \partial_{\kappa}\mathcal{L}(\kappa,u,-v), h\rangle_{[L^{\infty}(\Omega)]^{d\times d}}
$ due to the differentiability of $\kappa\mapsto v=v_\kappa$ (see Lemma \ref{L:diffv} below),
$\mathcal{E}_{\rm hom}(\kappa)=\mathcal{L}(\kappa,u,-v)$ and $\partial_{u}\mathcal{L}(\kappa,u,-v)=\partial_{w_\kappa}\mathcal{L}(\kappa,u,-v)=0$. Thus we obtain \eqref{eq:d-sensitivity}.
\end{proof}

\begin{rmk}[Existence and regularity of solutions to the adjoint equation]\label{hyp:u}
\rm
    The existence of a unique weak solution to \eqref{eq:adE} is assured by some natural assumption.
    Indeed, if the weak solution \(u\in K\) to \eqref{eq:rE} satisfies 
    \begin{equation}\label{assumption_u}
        u\in L^\infty(\Omega), \qquad u(x)\ge C>0 ~~(x\in\Gamma_0)
    \end{equation}
    for some \(C>0\) and \(\Gamma_0\subset\partial\Omega\) with \(|\Gamma_0|>0\), then we can deduce that \eqref{eq:adE} possesses a unique weak solution \(v\in H^1(\Omega)\).
    This result comes from the following inequality;
    there exists \(C>0\) such that
    \begin{equation*}
        \|v\|_{L^2(\Omega)}^2 \le C \left(  \|\nabla v\|_{L^2(\Omega)}^2 + \int_{\Gamma_0} \tilde{\kappa}v^2 \d \sigma   \right)
    \end{equation*}
    for any \(v\in H^1(\Omega)\), where \(\tilde{\kappa}\in L^{\infty}(\partial\Omega)\) which satisfies \(\tilde{\kappa}(x)\ge C>0\) (\(x\in \Gamma_0\)) for some \(C>0\) and \(\Gamma_0\in\partial\Omega\) with \(|\Gamma_0|>0\).
    The above inequality can be proved in a similar way to the proof of Lemma \ref{poincare} with slight modification.
    By virtue of the assumptions on \(u\) and this inequality, the usual method by Lax--Milgram theorem can be applied to \eqref{eq:adE} in order to show the existence of a weak solution.
    Unfortunately, it is difficult to prove the above assumption \eqref{assumption_u} on \(u\) rigorously.
    However, if \(\Omega\) and \(\kappa\) are sufficiently smooth, the solution \(u\) belongs to \(H^2(\Omega)\) (for detail, see \cite{B71}),
    and we can deduce that \(u\) satisfies the above conditions with \(\Gamma_0=\partial\Omega\).
    Therefore, our assumptions \eqref{assumption_u} are quite natural, and after this, we always impose \eqref{assumption_u} on the solution of \eqref{eq:rE} implicitly whenever we consider a solution \(v\) to \eqref{eq:adE}.
    Moreover, in this setting, we can derive \(v\in H^1(\Omega)\cap L^\infty(\Omega)\) and in particular \(v\in K\).
\end{rmk}

As for $u'=u_\kappa'$, we have the following
\begin{lem}[Differentiablity of $u$ with respect to $\kappa$] \label{L:diffu}
Suppose that \eqref{assumption_u}.
Then the nonnegative weak solution $\kappa\mapsto u=u_\kappa$ of \eqref{eq:rE} is differentiable at $\kappa$ and  
    $
    u_\kappa'h= \tilde{u} 
    $
    for the direction $h\in [L^{\infty}(\Omega)]^{d\times d}$. Here $\tilde{u} \in V$ satisfies
\begin{align}\label{u'eq}
   \int_{\Omega}\kappa(x)\nabla \tilde{u}(x)\cdot \nabla \varphi(x)\, \d x
    +
    \int_{\partial \Omega}\boldsymbol{\sigma}(r-1)u^{r-2}(x)\tilde{u}(x)\varphi(x)\, \d \sigma
    =
    -
   \int_{\Omega}h(x)\nabla u(x)\cdot \nabla \varphi(x)\, \d x
\end{align}
for all $\varphi\in V$.
\end{lem}

\begin{proof}
Let $\kappa[s]=\kappa+sh$ for $s>0$ and let $u[s]$ be a solution to \eqref{eq:rE} with $\kappa=\kappa[s]$.  
In this proof, we set $\boldsymbol{\sigma}=1$ for simplicity.
Differentiating with respect to $s>0$ in the weak form of \eqref{eq:rE} with $\kappa=\kappa[s]$, we have   
\begin{align*}
    \int_{\Omega}\kappa[s](x)\nabla u'[s](x)\cdot \nabla \varphi(x)\, \d x
    +
    \int_{\partial \Omega}(r-1)u^{r-2}[s](x)u'[s](x)\varphi(x)\, \d \sigma
    =
    -
   \int_{\Omega}h\nabla u[s](x)\cdot \nabla \varphi(x)\, \d x,
\end{align*}
which coincides with \eqref{u'eq} as $\tilde{u}=u'[s]$ and $s=0$.
Noting that $\kappa=\kappa[0]=\kappa[1]-h$, $u_{\kappa+h}=u[1]$ and $u_{\kappa}=u[0]$, we observe that, for any $\varphi\in V$,
\begin{align}
 \lefteqn{  \int_{\Omega}\kappa(x)\nabla(u_{\kappa+h}(x)-u_\kappa(x)-\tilde{u}(x))\cdot \nabla \varphi(x)\, \d x}
 \displaybreak[0]
\nonumber\\
&\quad+
    \int_{\partial \Omega}
    (u_{\kappa+h}^{r-1}(x)-u_\kappa^{r-1}(x)-(r-1)u_\kappa^{r-2}(x)\tilde{u}(x))\varphi(x)\, \d \sigma\displaybreak[0]\nonumber\\
&=
\int_{\Omega}[(\kappa[1](x)-h(x))\nabla u[1](x)-\kappa[0](x)\nabla u[0](x)-\kappa(x)\nabla \tilde{u}(x)]\cdot \nabla \varphi(x)\, \d x\displaybreak[0]\nonumber\\
&\quad+
    \int_{\partial \Omega}
    (u[1]^{r-1}(x)-u[0]^{r-1}(x)-(r-1)u[0]^{r-2}(x)\tilde{u}(x))\varphi(x)\, \d \sigma\displaybreak[0]\nonumber\\
& =
    -
   \int_{\Omega}h(x)\nabla (u[1](x)-u[0](x))\cdot \nabla \varphi(x)\, \d x
\le
    \|h\|_{L^{\infty}(\Omega)}\|\nabla(u_{\kappa+h}-u_\kappa)\|_{L^2(\Omega)}\|\nabla \varphi\|_{L^2(\Omega)}.
    \label{eq:lem3.4-1}
\end{align}
As for the upper bound of $\|\nabla(u_{\kappa+h}-u_\kappa)\|_{L^2(\Omega)}$, we deduce from the same argument that 
\begin{align}
\lefteqn{\int_{\Omega}\kappa(x)\nabla(u_{\kappa+h}(x)-u_\kappa(x))\cdot \nabla (u_{\kappa+h}(x)-u_\kappa(x))\, \d x
+
\int_{\partial \Omega}
  (u_{\kappa+h}^{r-1}(x)-u_\kappa^{r-1}(x))(u_{\kappa+h}(x)-u_\kappa(x))\, \d \sigma
}\nonumber\\
&=
  -
   \int_{\Omega}h(x)\nabla u_{\kappa+h}(x)\cdot \nabla (u_{\kappa+h}(x)-u_\kappa(x))\, \d x
    \le 
   \|h\|_{L^{\infty}(\Omega)}\|\nabla u_{\kappa+h}\|_{L^2(\Omega)}\|\nabla (u_{\kappa+h}-u_{\kappa})\|_{L^2(\Omega)},
   \label{eq:bounds-ukh}
\end{align}
which along with the boundedness of $(u_{k+h})$ in $V$ and the uniform ellipticity of $\kappa$ yields 
$$
\|\nabla(u_{\kappa+h}-u_\kappa)\|_{L^2(\Omega)}
\le C\|h\|_{L^{\infty}(\Omega)}.
$$
Hence, setting $\varphi=u_{\kappa+h}-u_\kappa-\tilde{u}$ in  \eqref{eq:lem3.4-1}, we have
\begin{align}
    &\|\nabla(u_{\kappa+h}-u_\kappa-\tilde{u})\|_{L^2(\Omega)}^2
    +
 \int_{\partial \Omega}
    (u_{\kappa+h}^{r-1}(x)-u_\kappa^{r-1}(x)-(r-1)u_\kappa^{r-2}(x)\tilde{u}(x))(u_{\kappa+h}(x)-u_\kappa(x)-\tilde{u}(x))\, \d \sigma\nonumber\\
    &\quad \le C\|h\|_{L^{\infty}(\Omega)}^2 \|\nabla(u_{\kappa+h}-u_\kappa-\tilde{u})\|_{L^2(\Omega)}.
    \label{eq:nablah}
\end{align}
By noting that the integrand of the second term of the left-hand side in \eqref{eq:nablah} is written as
\begin{align}
    \lefteqn
    {(u_{\kappa+h}^{r-1}-u_\kappa^{r-1}-(r-1)u_\kappa^{r-2}\tilde{u})(u_{\kappa+h}-u_\kappa-\tilde{u})}\nonumber\\
    &=
(u_{\kappa+h}^{r-1}-u_\kappa^{r-1}-(r-1)u_\kappa^{r-2}\tilde{u}-(u_{\kappa+h}-u_\kappa-\tilde{u}))(u_{\kappa+h}-u_\kappa-\tilde{u})
+
(u_{\kappa+h}-u_\kappa-\tilde{u})^2.\label{eq:boundaryest1-1}
\end{align}
The integral of the first term in the last line can be estimated as follows\/{\rm :}
\begin{align}
&\int_{\partial\Omega}
(u_{\kappa+h}^{r-1}(x)-u_\kappa^{r-1}(x)-(r-1)u_\kappa^{r-2}(x)\tilde{u}(x)-(u_{\kappa+h}(x)-u_\kappa(x)-\tilde{u}(x)))(u_{\kappa+h}(x)-u_\kappa(x)-\tilde{u}(x))\, \d \sigma\nonumber\\
&\quad \le 
C\|h\|_{L^{\infty}(\Omega)}^2 \|\nabla(u_{\kappa+h}-u_\kappa-\tilde{u})\|_{L^2(\Omega)}.
\label{eq:boundaryest1-2}
\end{align}
Hence one can derive that
\begin{align}
    \|\nabla(u_{\kappa+h}-u_\kappa-\tilde{u})\|_{L^2(\Omega)}\le C\|h\|_{L^{\infty}(\Omega)}^2,
\label{eq:nablauest}
\end{align}
which together with \eqref{eq:nablah}, \eqref{eq:boundaryest1-1} and \eqref{eq:boundaryest1-2} yields 
\begin{equation}
\|u_{\kappa+h}-u_\kappa-\tilde{u}\|_{L^2(\partial\Omega)}\le C\|h\|_{L^{\infty}(\Omega)}^2.
\label{eq:boundaryh}    
\end{equation}
Combining \eqref{eq:nablauest} with \eqref{eq:boundaryh}, we conclude that
$$
\lim_{\|h\|_{L^{\infty}(\Omega)}\to 0_+}\frac{\|u_{k+h}-u_k-\tilde{u}\|_{V}}{\|h\|_{L^{\infty}(\Omega)}}=0,
$$
which completes the proof.
\end{proof}

By the same argument as in the proof of Lemma \ref{L:diffu}, we have the following 
\begin{lem}[Differentiablity of $v$ with respect to $\kappa$] \label{L:diffv}
Suppose that \eqref{assumption_u}. 
Let $u\in K$ and $\tilde{u}\in V$ be weak solutions to \eqref{eq:rE} and \eqref{u'eq}, respectively.
Then the weak solution $\kappa\mapsto v=v_\kappa\in K$ of \eqref{eq:adE} is differentiable at $\kappa$ and  
    $
v_\kappa'h= \tilde{v} 
    $
    for the direction $h\in [L^{\infty}(\Omega)]^{d\times d}$. Here $\tilde{v}\in V$ satisfies
    \begin{align}
    \label{eq:adv}
   \lefteqn{\int_{\Omega}\kappa(x)\nabla \tilde{v}(x)\cdot \nabla \varphi(x)\, \d x}\nonumber\\
    &\quad +\boldsymbol{\sigma}
    \int_{\partial \Omega}
    [
    r(r-1)u^{r-2}(x)\tilde{u}(x)
    +(r-1)(r-2)u^{r-3}(x)\tilde{u}(x)v(x)
    +(r-1)u^{r-2}\tilde{v}(x)
    ]\varphi(x)\, \d \sigma\nonumber\\
    &=
    -
   \int_{\Omega}h\nabla v(x)\cdot \nabla \varphi(x)\, \d x
\end{align}
for all $\varphi\in V$.
\end{lem}

\begin{proof}
Let $\kappa[s]=\kappa+sh$ for $s>0$ and $\boldsymbol{\sigma}=1$ for simplicity. Let $u[s]$ and $v[s]$ be weak solutions to \eqref{eq:rE} with $\kappa=\kappa[s]$ and \eqref{eq:adE} with $\kappa=\kappa[s]$, respectively. Then, by differentiating with respect to $s>0$ in the weak form of \eqref{eq:adE}, we obtain \eqref{eq:adv} as $u'[s]=\tilde{u}$, $v'[s]=\tilde{v}$ and $s=0$. 
Furthermore, we get 
$
\|\nabla(v_{k+h}-v_k)\|_{L^2(\Omega)}\le C \|h\|_{L^{\infty}(\Omega)}$
as in \eqref{eq:bounds-ukh}.
Here we used the fact that $\|\nabla v_{k+h}\|_{L^2(\Omega)}\le C$ by virtue of \eqref{assumption_u} (see Remark \ref{hyp:u}).
Then we observe that, for any $\varphi\in V$,
\begin{align}
&\int_{\Omega}\kappa(x)\nabla(v_{\kappa+h}(x)-v_\kappa(x)-\tilde{v}(x))\cdot \nabla(v_{\kappa+h}(x)-v_\kappa(x)-\tilde{v}(x))\, \d x\displaybreak[0]\nonumber\\
&\quad+
(r-1)    \int_{\partial \Omega}[u_{\kappa+h}^{r-2}(x)
    v_{\kappa+h}(x)-u_\kappa^{r-2}(x)v_\kappa(x)-u_\kappa^{r-2}(x)\tilde{v}(x)](v_{\kappa+h}(x)-v_\kappa(x)-\tilde{v}(x))\, \d \sigma\displaybreak[0]\nonumber\\
&\quad-
\int_{\partial \Omega}[
 r(r-1)u_\kappa^{r-2}(x)\tilde{u}(x)
    +(r-1)(r-2)u_\kappa^{r-3}(x)\tilde{u}(x)v_\kappa(x)
](v_{\kappa+h}(x)-v_\kappa(x)-\tilde{v}(x))\, \d \sigma\displaybreak[0]\nonumber\\
&=
    -
   \int_{\Omega}h(x)\nabla (v_{k+h}(x)-v_k(x))\cdot \nabla (v_{k+h}(x)-v_k(x)-\tilde{v}(x))\, \d x\nonumber\\
&\le
    \|h\|_{L^{\infty}(\Omega)}\|\nabla(u_{\kappa+h}-u_\kappa)\|_{L^2(\Omega)}\|\nabla (v_{k+h}-v_k-\tilde{v})\|_{L^2(\Omega)}
    \le 
    C\|h\|_{L^{\infty}(\Omega)}^2\|\nabla (v_{k+h}-v_k-\tilde{v})\|_{L^2(\Omega)}.
\nonumber
\end{align}
Here we note that the integrand of the second line is written as
\begin{align*}
&[u_{\kappa+h}^{r-2}
    v_{\kappa+h}-u_\kappa^{r-2}v_\kappa-u_\kappa^{r-2}\tilde{v}](v_{\kappa+h}-v_\kappa-\tilde{v})\\
&\quad=
[(u_{\kappa+h}^{r-2}
    v_{\kappa+h}-u_\kappa^{r-2}v_\kappa-u_\kappa^{r-2}\tilde{v})
-(v_{\kappa+h}-v_\kappa-\tilde{v})](v_{\kappa+h}-v_\kappa-\tilde{v})
+
    (v_{\kappa+h}-v_\kappa-\tilde{v})^2.
\end{align*}
Thus we see by 
the same argument as in the proof of Lemma \ref{L:diffu} and 
the uniform ellipticity of $\kappa$ that
\begin{align}
\|
\nabla(v_{\kappa+h}-v_\kappa-\tilde{v})\|_{L^2(\Omega)}^2+
\|v_{\kappa+h}-v_\kappa-\tilde{v}\|_{L^2(\partial\Omega)}^2
\le 
C\|h\|_{L^{\infty}(\Omega)}^2
\|\nabla (v_{k+h}-v_k-\tilde{v})\|_{L^2(\Omega)},
\nonumber
\end{align}
which implies that 
$
\|
v_{\kappa+h}-v_\kappa-\tilde{v}\|_{V}\le 
C\|h\|_{L^{\infty}(\Omega)}^2$, and hence,
$$
\lim_{\|h\|_{L^{\infty}(\Omega)}\to 0_+}\frac{\|v_{k+h}-v_k-\tilde{v}\|_{V}}{\|h\|_{L^{\infty}(\Omega)}}=0.
$$
This completes the proof.
\end{proof}

We next observe the relation between $u$ and $v$.
\begin{prop}[Difference in the gradients of state and adjoint equations]
\label{P:state-ad}
Assume that \eqref{assumption_u}.
Let $u\in K$ and $v\in V$ be weak solutions to \eqref{eq:rE} and \eqref{eq:adE}, respectively.
Then it holds that
\begin{equation*}
\|\nabla (u-v)\|_{L^2(\Omega)}^2\le 
\frac{r-1}{\alpha'} \sqrt{\frac{2}{\alpha_0}} \|f\|_{L^2(\Omega)}
\left\{ \frac{1}{2\alpha_0} \|f\|_{L^2(\Omega)}^2 + \frac{r-2}{2} \boldsymbol{\sigma} |\partial \Omega| \right\}^{\frac{1}{2}},
\end{equation*}
where \(\alpha_0 := \min(\alpha',\frac{r}{2}\boldsymbol{\sigma})\). 
\end{prop}

\begin{proof}
By \eqref{eq:rE} and  \eqref{eq:adE}, we have
\begin{equation*}
    \int_{\Omega} \kappa(x) \nabla(u(x)-v(x)) \cdot \nabla(u(x)-v(x)) \d x - \int_{\partial \Omega} \kappa(x) \nabla (u(x)-v(x)) \cdot \nu(x) (u(x)-v(x)) \, \d \sigma=0.
\end{equation*}
From the boundary conditions, the second term of the left-hand side implies that
\begin{align}
\lefteqn{
-\int_{\partial \Omega} \kappa(x) \nabla (u(x)-v(x)) \cdot \nu(x) (u(x)-v(x))\, \d\sigma}\nonumber\\
&=  
\int_{\partial \Omega} \boldsymbol{\sigma} \left\{ u^{r-1}(x) -(r-1) u^{r-2}(x)v(x) -ru^{r-1}(x) \right\} (u(x)-v(x))\, \d\sigma
\displaybreak[0]\nonumber\\[2mm]
&= 
- \boldsymbol{\sigma} (r-1) \int_{\partial \Omega} u^{r-2}(x)(u^2(x)-v^2(x))\, \d\sigma
\ge 
- \boldsymbol{\sigma} (r-1) \int_{\partial \Omega} u^{r}(x)\, \d\sigma.\nonumber
\end{align}
Here we used the nonnegativity of $u\ge0$ in the last inequality.
Hence we have
\begin{equation}\label{eq:u-v}
    \alpha' \|\nabla(u-v)\|_{L^2(\Omega)}^2 \le \boldsymbol{\sigma} (r-1) \int_{\partial \Omega} u^r(x)\, \d\sigma
    \overset{\eqref{eq:rE}}{\le} (r-1)\int_{\Omega} fu\, \d x
    \le (r-1)\|f\|_{L^2(\Omega)}\|u\|_{V}.
\end{equation}
In the rest of the proof, it suffices to show that
\begin{equation}\label{est:H^1}
    \|u\|_V \le \sqrt{\frac{2}{\alpha_0}} 
\left\{ \frac{1}{2\alpha_0} \|f\|_{L^2(\Omega)}^2 + \frac{r-2}{2} \boldsymbol{\sigma} |\partial \Omega| \right\}^{\frac{1}{2}}.
\end{equation}
By \eqref{eq:rE}, it holds that
\begin{equation*}
    \alpha' \|\nabla u\|_{L^2(\Omega)}^2 + \boldsymbol{\sigma} \int_{\partial \Omega} u^r(x) \, \d\sigma \le \int_{\Omega} fu\, \d x.    
\end{equation*}
Moreover, it follows from H\"older's inequality and Young's inequality that
\begin{equation*}
    \int_{\partial \Omega} u^2(x)\, \d\sigma 
    \le 
    \left( \int_{\partial \Omega} u^r(x)\, \d\sigma \right)^{\frac{2}{r}} |\partial \Omega|^{\frac{r-2}{r}}
    \le \frac{2}{r} \int_{\partial \Omega} u^r(x) \, \d\sigma + \frac{r-2}{r} |\partial \Omega|.
\end{equation*}
Thus we obtain
\begin{equation*}
    \alpha'\|\nabla u\|_{L^2(\Omega)}^2 + \frac{r}{2} \boldsymbol{\sigma} \|u\|_{L^2(\partial \Omega)}^2 - \frac{r-2}{2}\boldsymbol{\sigma} |\partial \Omega| \le \|f\|_{L^2(\Omega)}\|u\|_{V}.
\end{equation*}
Using Young's inequality, we can derive
\begin{equation*}
    \alpha_0 \|u\|_V^2 \le \frac{\alpha_0}{2} \|u\|_V^2 + \frac{1}{2\alpha_0} \|f\|_{L^2(\Omega)}^2 + \frac{r-2}{2}\boldsymbol{\sigma} |\partial \Omega|,
\end{equation*}
which implies \eqref{est:H^1}.
Therefore, by \eqref{eq:u-v} and \eqref{est:H^1}, we obtain 
\begin{equation*}
    \|\nabla (u-v)\|_{L^2(\Omega)}^2\le 
\frac{r-1}{\alpha'} \sqrt{\frac{2}{\alpha_0}} \|f\|_{L^2(\Omega)}
\left\{ \frac{1}{2\alpha_0} \|f\|_{L^2(\Omega)}^2 + \frac{r-2}{2} \boldsymbol{\sigma} |\partial \Omega| \right\}^{\frac{1}{2}},
\end{equation*}
which is the desired result.
\end{proof}
Combining Propositions \ref{P:d-sensitivity} with \ref{P:state-ad}, we see by $f\not\equiv 0$ 
that, for any $\kappa, h\in \mathcal{M}(\alpha',\beta')$,
\begin{align*}
    \left\langle \mathcal{E}_{\rm hom}'(\kappa),h\right\rangle_{[L^{\infty}(\Omega)]^{d\times d}}
  &  =
  -  \int_{\Omega}h(x)\nabla u(x)\cdot \nabla v(x)\, \d x\\
&= 
\frac{1}{2}\int_{\Omega}h(x)\nabla (u(x)-v(x))\cdot \nabla(u(x)- v(x))\, \d x\\
&\quad
-
\frac{1}{2}\int_{\Omega}h(x)\nabla u(x)\cdot \nabla u(x)\, \d x
-
\frac{1}{2}\int_{\Omega}h(x)\nabla v(x)\cdot \nabla v(x)\, \d x\\
&<
\frac{\beta'}{2}\|\nabla (u-v)\|_{L^2(\Omega)}^2\\
& \le
\frac{\beta'}{2}\frac{r-1}{\alpha'} \sqrt{\frac{2}{\alpha_0}} \|f\|_{L^2(\Omega)}
\left\{ \frac{1}{2\alpha_0} \|f\|_{L^2(\Omega)}^2 + \frac{r-2}{2} \boldsymbol{\sigma} |\partial \Omega| \right\}^{\frac{1}{2}}.
\end{align*}
Thus one may expect  
$
\mathcal{E}_{\rm hom}'(\kappa)=-\nabla u\cdot \nabla v\le 0
$ 
a.e.~in $\Omega$ at least in the case where $f\ge 0$ is small (see Remark \ref{R:linear-selfad} below). 
In this particular case, due to $\mathcal{E}(\theta)\le \mathcal{E}_{\rm hom}(\kappa)$ by \eqref{eq:est-hom}, one can estimate the minimum value by replacing \eqref{eq:RP} with the following minimization problem\/{\rm :} 
\begin{align}\label{eq:RP2}
\inf_{\theta\in \Theta} \mathcal{E}(\theta),    
\end{align}
where $\Theta:=\{\theta\in L^{\infty}(\Omega;[0,1])\colon \|\theta\|_{L^{1}(\Omega)}=\gamma|\Omega|\}$\footnote{In the self adjoint problem (i.e.,~$v=\pm u$), there is a case where an optimal homogenized matrix can be characterized as $\kappa^\ast=\kappa[\theta^\ast]=\alpha(1-\theta^\ast)+\beta\theta^\ast$. Here $\theta^\ast$ is an optimal volume fraction (i.e.,~$\mathcal{E}_{\rm hom}(\kappa^\ast)=\mathcal{E}_{\rm hom}(\theta^\ast,\kappa_{\rm hom}^\ast)=\mathcal{E}(\theta^\ast)$). Hence it suffices to consider \eqref{eq:RP2} instead of \eqref{eq:RP}; however, nonlinear problems cause non-self adjoint problems in general (see \cite[Theorem 5.5]{ACMOY19} for an optimal homogenized flux). 
Thus the problem with $\chi_{\Omega_1}$ being replaced by $\theta$ as in \eqref{eq:RP2} is just a problem to estimate the infimum value in general. 
On the other hand, as in the homogeneous Dirichlet boundary condition, one can construct the self adjoint problem for the homogeneous Robin boundary condition (i.e.,~$r=2$) by setting $\mathcal{E}(\chi)=\langle f,u_{\chi}\rangle$. 
}

Now, we are in a position to describe a numerical algorithm for the volume fraction $\theta\in \Theta$.
Based on the (steepest gradient) descent method (or time-discrete version of the gradient flow) and Proposition \ref{P:d-sensitivity}, we set 
\begin{align}
    \theta_{i+1}=\theta_{i}-\tau \mathcal{E}'(\theta_i)=\theta_{i}-\tau (\beta-\alpha)\nabla u_{\theta_i}\cdot \nabla v_{\theta_i}
    \quad \text{ for } i\in\N\cup\{0\}.
\label{eq:theta_i+1}
\end{align}
Here $\theta_0\in L^{\infty}(\Omega)$ is an initial volume fraction,
$\tau>0$ stands for the step width (or time step, i.e.,~$\theta_i$ implies $\theta_i=\theta(x,\tau i)$) and
$u_\theta\in K$ and $v_\theta\in K$ are unique weak solutions to \eqref{eq:rE} with $\kappa=\kappa[\theta]$ and \eqref{eq:adE} with $\kappa=\kappa[\theta]$, respectively. 
Repeating \eqref{eq:theta_i+1} until $\|\theta_{i+1}\|_{L^1(\Omega)}=\gamma|\Omega|$ and $\|\theta_{i+1}-\theta_i\|_{L^1(\Omega)}\le \eta$ for $\eta>0$ small enough, one can estimate the minimum value of $\mathcal{E}(\chi_{\Omega_1})$ numerically by Theorem \ref{T:existence}. The following is the numerical algorithm\/{\rm :}

\begin{algorithm}[H]
    \caption{Optimization for the volume fraction of \eqref{eq:RP2}. 
    }
    \label{alg1}
    \begin{algorithmic}[1]
    \STATE 
    Let $i=0$.
    Set $\Omega\subset \R^d$, $\alpha,\beta, \gamma,\tau>0$, $f\in L^2(\Omega;\R_+)$ and $\theta_0\in \Theta$.
    \STATE
    Solve \eqref{eq:rE} with $\kappa=\kappa[\theta_i]$ to determine $u_{\theta_{i}}$ in \eqref{eq:theta_i+1}. 
    \STATE
    Solve \eqref{eq:adE} with $\kappa=\kappa[\theta_i]$ to determine $v_{\theta_{i}}$ in \eqref{eq:theta_i+1}.  
    \STATE
    Compute \eqref{eq:theta_i+1}.  
  \STATE
Determine $\lambda\in\R$ such that 
$$
|\gamma|\Omega|-\|\theta^\lambda_{i+1}\|_{L^1(\Omega)}|\le \eta_1,
$$
where $\eta_1>0$, $\theta^\lambda_{i+1}$ is such that
$$
\theta^\lambda_{i+1}=\max\{0,\min\{\theta_{i+1}+\lambda,1\}\} 
$$
(see, e.g.,~\cite[\S 3.5]{ACMOY19} for projected gradient methods).

\STATE
   Check for the convergence condition, 
   \begin{align}
       \|\theta^\lambda_{i+1}-\theta_{i}\|_{L^1(\Omega)}\le \eta_2,
       \label{eq:cc}
   \end{align}
   where $\eta_2>0$.
   If it is satisfied, then terminate the optimization as $\theta_{i+1} \leftarrow \theta^\lambda_{i+1}$; otherwise, return 2 after setting $\theta_{i} \leftarrow \theta^\lambda_{i+1}$. 
    \end{algorithmic}
\end{algorithm}

\begin{rmk}[Linearization of the thermal radiation boundary condition] \label{R:Newton}
      \rm 
To solve \eqref{eq:rE} with $\kappa=\kappa[\theta]$ numerically,  
we first approximate $\boldsymbol{\beta}(u_\theta)$ as $\boldsymbol{\beta}(u_{\rm old})+\boldsymbol{\beta}'(u_{\rm old})(u_\theta-u_{\rm old})$ as in the Newton--Raphson method. 
Here $u_{\rm old}$ is an arbitrarily given function, and we choose $r=d+2$ in \eqref{eq:rE} based on \cite{LV89}, that is, 
$
\boldsymbol{\beta}(u_\theta)\approx \boldsymbol{\sigma}u_{\rm old}^{d+1}+
\boldsymbol{\sigma}(d+1)u_{\rm old}^{d}(u_\theta-u_{\rm old})
=
\boldsymbol{\sigma}(d+1)u_{\rm old}^{d}u_\theta-\boldsymbol{\sigma} du_{\rm old}^{d+1}
$.
We next solve the following linearized equation\/{\rm :} 
\begin{align}
0&=-\int_\Omega f(x) \varphi(x)\, \d x
+
\int_\Omega \kappa[\theta](x) \nabla u_\theta(x)\cdot \nabla \varphi(x)\, \d x\nonumber\\
&\quad+ 
\int_{\partial \Omega}    
\underbrace{\bigl[\boldsymbol{\sigma} (d+1)u_{\rm old}^d (x)u_\theta(x)
    -
    \boldsymbol{\sigma} du_{\rm old}^{d+1}(x)\bigl]}_{\text{linearized thermal radiation boundary condition}}  
    \varphi(x)\, \d \sigma
\quad \text{ for all $\varphi\in V$. }
    \label{eq:L-state}
\end{align}
We finally check the following convergence condition\/{\rm :}
\begin{align}
\left|
    \int_{\Omega} \kappa[\theta](x) |\nabla u_\theta(x)|^2\, \d x+ 
    \boldsymbol{\sigma}\int_{\partial \Omega} |u_\theta (x)|^{d+2}\, \d \sigma
    -\int_\Omega f(x) u_\theta(x)\, \d x    
\right|\le \eta_3
\quad \text{ for some $\eta_3>0$.} 
\label{eq:L-conv}
\end{align}
If \eqref{eq:L-conv} is not satisfied, we set $u_{\rm old}\leftarrow u_\theta$, and then we solve \eqref{eq:L-state} again. 
This procedure is repeated until \eqref{eq:L-conv}
is satisfied. 
\end{rmk}

\begin{rmk}[Convergence condition]
\rm 
By \eqref{eq:theta_i+1}, it holds that
$$
|\theta_{i+1}-\theta_{i}|=\tau (\beta-\alpha)|\nabla u_{\theta_i}\cdot \nabla v_{\theta_i}|.
$$
If $\theta_{i+1}$ attains the critical point of $\mathcal{E}(\theta)$, then the right-hand side vanishes.  
Since it belongs to $L^1(\Omega)$ at least, the convergence condition \eqref{eq:cc} is reasonable.
In this paper, we do not mention the regularization of sensitivity $\mathcal{E}'(\theta_i)$ to become $\theta_{i+1}\in L^{\infty}(\Omega)$ since Algorithm \ref{alg1} is only used to estimate the minimum value of the original optimal design problem \eqref{eq:OP} and $\inf_{\theta\in \Theta} \mathcal{E}(\theta)= \inf_{\theta\in \tilde{\Theta}} \mathcal{E}(\theta)$,    
where $\tilde{\Theta}:=\{\theta\in L^{1}(\Omega)\colon \theta(x)\in [0,1] \text{ and } \|\theta\|_{L^{1}(\Omega)}=\gamma|\Omega|\}$.
\end{rmk}

\begin{rmk}[Self adjointness and convexity of a linearized problem] \label{R:linear-selfad}
\rm
Let $u_{\rm L}$ be a solution to \eqref{eq:rE} 
with $|u|^{r-2}u$ being replaced by $u_{\rm old}^{r-1}$, and then consider the minimization problem \eqref{eq:RP} 
with $u_{\rm hom}$ being replaced by $u_{\rm L}$. Here $u_{\rm old}$ is a function that appears in Remark \ref{R:Newton}.
Then it can be regarded as a self-adjoint problem by the same argument as in Proposition \ref{P:d-sensitivity}. Thus \eqref{eq:RP2} with $u_{\rm hom}$ being replaced by $u_{\rm L}$ turns out to be a true relaxation problem by Theorem \ref{T:existence}, Proposition \ref{P:d-sensitivity} and \ref{eq:est-hom}, and moreover, it has only global minimizers in terms of double minimization; indeed, define $\tilde{\mathcal{E}}\colon V\to \R$ by
$$
\tilde{\mathcal{E}}(w)=
 \frac{1}{2}\int_\Omega \kappa[\theta](x)|\nabla w(x)|^2\, \d x
    +
    \int_{\partial \Omega} \boldsymbol{\sigma}u_{\rm old}^{r-1}(x) w(x)\, \d \sigma
   -\int_\Omega f(x) w(x)\, \d x.
$$
Then we see that $u_{\rm L}={\rm argmin}_{w\in K}\tilde{\mathcal{E}}(w)$ and  
\begin{align*}
\int_{\Omega} \kappa[\theta](x)|\nabla u_{\rm L}(x)|^2\, \d x
&=
2\left(\int_\Omega f(x) u_{\rm L}(x)\, \d x    
-\int_{\partial \Omega} \boldsymbol{\sigma}u_{\rm old}^{r-1}(x) u_{\rm L}(x)\, \d \sigma\right)
-
\int_{\Omega} \kappa[\theta](x)|\nabla u_{\rm L}(x)|^2\, \d x\\ 
&= -2 \inf_{w\in K}\tilde{\mathcal{E}}(w).
\end{align*}
Let $\hat{\mathcal{E}}:V\times [L^2(\Omega)]^d\to \R$ be such that $\hat{\mathcal{E}}(w,\nabla w)=\tilde{\mathcal{E}}(w)$. 
Since $\hat{\mathcal{E}}$ is convex, the dual energy yields 
$$
\inf_{w\in K}\tilde{\mathcal{E}}(w)=\inf_{\substack{P^\ast\in [L^2(\Omega)]^d,\\ -\dv P^\ast=f \text{ in }\ \Omega, \\
-P^\ast\cdot \nu=\boldsymbol{\sigma}u_{\rm old}^{r-1} \text{ on }\ \partial\Omega
}}\int_{\Omega} \kappa[\theta]^{-1}(x)|P^\ast(x)|^2\, \d x
$$
(see, e.g.,~\cite[Theorem 2.29 and Example 2.30]{ACMOY19}).
Thus the minimization problem \eqref{eq:RP} with $u_{\rm hom}$ being replaced by $u_{\rm L}$ is equivalent to the following double minimization problem\/{\rm :}
\begin{align}\label{eq:doublemin}
    \min_{(\theta,P^\ast)\in \mathcal{W}} \int_{\Omega} \kappa[\theta]^{-1}(x)|P^\ast(x)|^2\, \d x, 
\end{align}
where 
$$
\mathcal{W}:=\{
(\theta,P)\in \Theta \times [L^2(\Omega)]^d\colon
-\dv P^\ast=f \text{ in }\ \Omega \text { and } 
-P^\ast\cdot \nu=\boldsymbol{\sigma}u_{\rm old}^{r-1} \text{ on }\ \partial\Omega
\}.
$$
Since $\mathcal{W}$ is convex, and $(\theta, P^\ast)\mapsto \kappa[\theta]^{-1}|P^\ast|^2$ is also convex, the assertion is obtained. 
Hence, if $u_{\rm old}^{r-1}$ sufficiently approximates $u^{r-1}$ on $\partial\Omega$, the convergence value of energies via Algorithm \ref{alg1} also approximates the minimum value for \eqref{eq:OP} with 
$A_{\chi_{\Omega_1}}=\kappa[\chi_{\Omega_1}]$ and
$\boldsymbol{\beta}(u_{\chi_{\Omega_1}})=\boldsymbol{\sigma}u_{\chi_{\Omega_1}}^{r-1}$.

In this paper, to estimate the minimum value of \eqref{eq:OP} with $A_{\chi_{\Omega_1}}=\kappa[\chi_{\Omega_1}]$ and $\boldsymbol{\beta}(u_{\Omega_1})=\boldsymbol{\sigma}u_{\Omega_1}^{d+1}$ numerically, we consider the state equation as an approximated equation with inhomogeneous Neumann boundary conditions in optimization of the volume fraction; in other words, $\mathcal{E}'(\theta_i)$ in \eqref{eq:theta_i+1} is regarded as $\mathcal{E}'(\theta_i)=-(\beta-\alpha)|\nabla u_{\theta_i}|^2$.
\end{rmk}

\section{Approximation problem for \eqref{eq:OP} via positive parts of level set functions}\label{S:LSM}
In this section, we shall prepare a numerical analysis to find 
two-material distributions that give a value close to the minimum for \eqref{eq:OP} with 
$A_{\chi_{\Omega_1}}=\kappa[\chi_{\Omega_1}]$ and 
$\boldsymbol{\beta}(w)=\boldsymbol{\sigma}w^{d+1}$. 
As already mentioned in (ii) of Remark \ref{R:opt-d}, 
we need to construct the optimal volume fraction $\theta$ numerically such that
the intermediate set $[0<\theta<1]$ rarely appears due to non-existence of minimizers for \eqref{eq:OP} in general. 
As one of the methods to avoid the so-called grayscale problem, 
\emph{level set methods} (see, e.g.,~\cite{OS88,AJT02,AJT04,AA06}) are known and employed to construct an approximated minimizer below.
In level set methods, the following level set function is introduced to represent two-material domains\/{\rm :}  
\begin{align*}
\phi(x)
\begin{cases}
>0,\quad & x\in \Omega_1,\\
=0,\quad & x\in \partial\Omega_1\cap\partial\Omega_0,\\
<0,\quad & x\in \Omega_0.
\end{cases}
\end{align*}
Based on \cite{O23}, we consider the following perimeter constraint problem via the positive part of the level set function as an approximation problem of \eqref{eq:RP2}\/{\rm :}
\begin{align}\label{eq:P}
    \inf_{\phi\in U_{\rm ad}} \left\{J_\e(\phi):=\mathcal{E}(\phi_+)+\frac{\e}{p}\int_\Omega |\nabla \phi(x)|^p\, \d x\right\},
\end{align}
where $\phi_+=\max\{0,\phi\}$, $U_{\rm ad}:=\{ \phi\in W^{1,p}(\Omega)\colon |\phi|\le 1 \text{ and } \|\phi_+\|_{L^1(\Omega)}=\gamma |\Omega|\}$, $1<p<+\infty$ and $\e>0$.
In particular, the second term of $J_\e$ (i.e.,~the $p$-Dirichlet energy) plays a role of perimeter constraint (cf.~\cite{AAC86,BMW06}).

\subsection{Characterization of minimizers for level set functions}
In order to form the basis of numerical analysis for \eqref{eq:P}, we first show the following
\begin{thm}[Existence theorem for minimizers of \eqref{eq:P}]\label{T:existence2}
    There exists at least one minimizer of \eqref{eq:P}.
\end{thm}

\begin{proof}
Let $(\phi^n)$ be a minimizing sequence in $U_{\rm ad}$. Thus $\phi^n$ satisfies
$$
\lim_{n\to +\infty} J_\e(\phi^n)=\inf_{\phi\in U_{\rm ad}}J_\e(\phi).
$$
Since $(\phi^n)$ is bounded in $W^{1,p}(\Omega)$ due to $J_\e(\phi^n)\to +\infty$ as $\|\phi^n\|_{W^{1,p}(\Omega)}\to +\infty$, there exist a (not relabeled) subsequence of $(n)$ and $\phi^\ast\in U_{\rm ad}$
such that
\begin{align}\label{eq:conv1}
    \phi^n\to \phi^\ast \text{ weakly in } W^{1,p}(\Omega)
 \end{align}
and 
\begin{align}\label{eq:conv2}
    \phi^n_+\to \phi^\ast_+ \text{ a.e.~in } \Omega.
\end{align}
Hence Theorem \ref{T:H-conv} ensures that
\begin{align}
    \lim_{n\to +\infty}\mathcal{E}(\phi^n_+)= \mathcal{E}(\phi^\ast_+).
\label{eq:econv}
\end{align}
Combining \eqref{eq:econv} with the weak lower semicontinuity of norm, we obtain
$$
\inf_{\phi\in U_{\rm ad}}  J_\e(\phi)
\le
J_\e(\phi^\ast)
\le \liminf_{n\to +\infty} J_\e(\phi^n)
=
\inf_{\phi\in U_{\rm ad}}  J_\e(\phi),
$$
which completes the proof.
\end{proof}

Furthermore, we have the following
\begin{thm}[Convergence of functionals for minimizers]\label{T:conv-min}
Let $\phi^\e$ be a minimizer of \eqref{eq:P}. 
Then there exist a {\rm(}not relabeled{\rm)} subsequence of $(\e)$ and $\phi^\ast\in U_{\rm ad}$ such that
$\phi^\e\to \phi^\ast$ weakly in $W^{1,p}(\Omega)$ and
$$
\mathcal{E}(\phi^\ast_+)=
\lim_{\e\to 0_+}J_\e(\phi^\e)=\inf_{\phi\in U_{\rm ad}}\mathcal{E}(\phi_+).
$$
\end{thm}

\begin{proof}
We first note that $a_\e:=J_{\e}(\phi^\e)$ has a limit; indeed, if $s\ge t$ for $s,t>0$, we have
\begin{align}
a_{s}\ge J_{t}(\phi^{s})\ge a_{t}\ge \inf_{\theta\in\Theta} \mathcal{E}(\theta)\ge 0,   
\label{eq:necessary}    
\end{align}
which yields the assertion. 
Furthermore, due to $a_\e\to +\infty$ as $\|\phi^\e\|_{W^{1,p}(\Omega)}\to +\infty$, we see that $(\phi^\e)$ is bounded in $W^{1,p}(\Omega)$. Thus, as in the proof of Theorem \ref{T:existence2}, there exist a (not relabeled) subsequence of $(\e)$ and $\phi^\ast\in U_{\rm ad}$ such that \eqref{eq:conv1} and \eqref{eq:conv2} with $n=\e$. 
Therefore, it follows that
\begin{align*}
    \lim_{\e\to 0_+} a_{\e}=\mathcal{E}(\phi^\ast_+).  
\end{align*}
Noting that, for any $\phi\in U_{\rm ad}$, 
\begin{align*}
    a_{\e} 
    &\le \int_{\Omega} \kappa[\phi_+](x)|\nabla u_{\phi}(x)|^2\, \d x
    +\frac{\e}{p}\int_\Omega|\nabla \phi(x)|^p\, \d x
    \to \mathcal{E}(\phi_+)\quad \text{ as  }\ \e\to 0_+,
\end{align*}
one obtains
$\mathcal{E}(\phi^\ast_+)\le \mathcal{E}(\phi_+)$ for all $\phi\in U_{\rm ad}$. This completes the proof.
\end{proof}

\begin{rmk}[Approximate solutions for \eqref{eq:OP}]\label{R:H1}
\rm
By Theorem \ref{T:conv-min}, it holds that
\begin{align*}
    \mathcal{E}(\phi^\ast_+)\le \mathcal{E}(\theta)\quad  \text{ for all  }\ \theta\in \Theta\cap W^{1,p}(\Omega).
\end{align*}
Thus $\phi_+^\ast$ turns out to be a minimizer of \eqref{eq:RP2} under
 $
\inf_{\theta\in\Theta} \mathcal{E}(\theta)
=
\inf_{\theta\in\Theta\cap W^{1,p}(\Omega)} \mathcal{E}(\theta)
$. 
In this case, Theorem \ref{T:existence} ensures that $\phi_+^\e$ for $\e>0$ small enough can be regarded as an approximate solution for \eqref{eq:OP} under the optimal homogenized matrix can be written as the upper bound.
Since $\chi_{\Omega_1}\in BV(\Omega)$ is required at least in the perimeter constraint problem for \eqref{eq:OP}, it is reasonable to assume additional regularity as a setting that avoids the grayscale problem. 
In particular,
the optimal volume fraction $\theta^\ast$ of \eqref{eq:RP2} is weakly differentiable in the direction of $\nabla u_{\theta^\ast}$ under $u_{\theta^\ast}\in H^2(\Omega)$; indeed, we observe that, for any $\varphi\in C^{\infty}_{\rm c}(\Omega)$,
\begin{align*}
( F,\varphi)_{L^2(\Omega)}&:=-\int_{\Omega} \theta^\ast(x)\dv [(\nabla u_{\theta^\ast}(x))\varphi(x)]\, \d x\\
&\ =
-\int_{\Omega} \theta^\ast(x)\Delta  u_{\theta^\ast}(x) \varphi(x)\, \d x
-\int_{\Omega} \theta^\ast(x) \nabla u_{\theta^\ast}(x)\cdot \nabla \varphi(x)\, \d x.
\end{align*}
Then the second term in the last line is written as
\begin{align*}
    -\int_{\Omega} \theta^\ast(x) \nabla u_{\theta^\ast}(x)\cdot \nabla \varphi(x)\, \d x
=
\frac{1}{\beta-\alpha}\int_{\Omega} [\alpha \nabla u_{\theta^\ast}(x)\cdot \nabla \varphi(x)-f(x)\varphi(x)]\, \d x,
\end{align*}
and therefore, $(\nabla \theta^{\ast}\cdot\nabla u_{\theta^\ast},\varphi)_{L^2(\Omega)}$ makes sense
by noting that
\begin{align*}
|( F,\varphi)_{L^2(\Omega)}|
=
\left|-\int_{\Omega} \left[\left(\theta^\ast(x)+\frac{\alpha}{\beta-\alpha}\right)\Delta  u_{\theta^\ast}(x)+ 
\frac{1}{\beta-\alpha}f(x)\right]\varphi(x)\, \d x\right|\le C_{\theta^\ast, f}\|\varphi\|_{L^2(\Omega)}
\end{align*}
(see \cite{D15} for the homogeneous Dirichlet boundary condition).
\end{rmk}

\begin{rmk}[Extension from $\phi_+$ to $\phi_+^m$]\label{R:extension}
\rm 
In Theorems \ref{T:existence2} and \ref{T:conv-min}, one can replace $\phi_+$ in \eqref{eq:P} with $\phi_+^m$ for all $m\ge 1$. Indeed, let $A:L^{m+1}(\Omega)\to L^{(m+1)/m}(\Omega)$ be an operator defined by $A(w)= w_+^m$.
Then $A$ is maximal monotone in $L^{m+1}(\Omega)\times L^{(m+1)/m}(\Omega)$. Noting that
\begin{align*}
\phi^n&\to \phi^\ast &&\text{ weakly in }\ L^{m+1}(\Omega),\\
&\ && \text{ strongly in }\ L^{p}(\Omega),\\
 A(\phi^n) &\to \xi\quad &&\text{ weakly in }\ L^{p/(p-1)}(\Omega)\cap L^{(m+1)/m}(\Omega) 
\end{align*}
for some $\xi\in L^{p/(p-1)}(\Omega)\cap L^{(m+1)/m}(\Omega)$ and $n=n_k$, we have
$$
\langle A(\phi^n),\phi^n\rangle_{L^{m+1}(\Omega)}=
\langle A(\phi^n),\phi^n\rangle_{L^{p}(\Omega)}
\to 
\langle \xi,\phi^\ast\rangle_{L^{p}(\Omega)}
=
\langle \xi,\phi^\ast\rangle_{L^{m+1}(\Omega)}
\quad \text{ as }\ n\to +\infty .
$$
Hence Minty’s trick (see e.g.,~\cite[Corollary 2.4.]{B10}) ensures that $\xi=A(\phi^\ast)=(\phi_+^\ast)^m$. 
Furthermore, 
since $(\phi_+^n)$ is also bounded in $W^{1,p}(\Omega)$, 
one can extract a subsequence of $(n_k)$ (still denoted by $n_k$) such that $\phi_+^{n_k}\to \phi_+^\ast$ a.e.~in $\Omega$, which along with the boundedness of $(\phi_+^n)$ in $L^{\infty}(\Omega)$ yields $A(\phi^{n_k})\to A(\phi^\ast)$ strongly in $L^\ell(\Omega)$ for all $\ell\ge 1$.
The rest of the proofs runs as before. 

In this paper, we select $m=1$ to compare the results in \cite{O23} (cf.~\cite{LZ11} for $m>1$). 
\end{rmk}

\subsection{Numerical algorithm for optimization of level set functions}
Before describing the numerical algorithm, we derive the equation to update the level set function. 
As in \eqref{eq:theta_i+1}, we introduce the following (gradient) descent method\/{\rm :}
$$
\phi_{i+1}=\phi_i-\tau J_\e'(\phi_i)\quad \text{ for } i\in \N\cup\{0\}.
$$
Since the Fr\'echet derivative of $p$-Dirichet energy for the level set function is $-\Delta_p\phi=-\dv(|\nabla \phi|^{p-2}\nabla \phi)$ under homogeneous Dirichlet/Neumann boundary condition, we have
\begin{align}
    \phi_{i+1}=\phi_i-\tau (\partial_{\phi_i}\mathcal{E}((\phi_i)_+)-\e\Delta_p \phi_{i+1})\quad \text{ for } i\in \N\cup\{0\}.
\label{eq:GD1}
\end{align}
Here we note that $\Delta_p \phi_{i}$ is replaced with $\Delta_p \phi_{i+1}$ in order to satisfy $\phi_{i+1}\in W^{1,p}(\Omega)$
(see, e.g.~\cite[Corollary 27.9]{BC11}) for forward-backward splitting schemes).
Furthermore, by Remark \ref{R:extension} and Proposition \ref{P:d-sensitivity}, the above update equation
\eqref{eq:GD1} with $(\phi_i)_+$ being replaced by $(\phi_i)_+^m$ can be written as 
\begin{align}
    \phi_{i+1}=\phi_i-\zeta[\phi_i](-m(\beta-\alpha)(\phi_i)_+^{m-1}\chi_{\phi_i}\nabla u_{\phi_i}\cdot \nabla v_{\phi_i} -\e\Delta_p \phi_{i+1})
\quad \text{ for } i\in \N\cup\{0\},
\label{eq:NLD2}
\end{align}
where $u_{\phi_i}$ and $v_{\phi_i}$ are the unique weak solutions to \eqref{eq:rE} with $\kappa=\kappa[(\phi_i)_+]$ and \eqref{eq:adE} with $\kappa=\kappa[(\phi_i)_+]$, respectively.
Here we note that $\tau>0$ in \eqref{eq:GD1} is extended to 
$\zeta[\phi_i]\ge 0$, which is the generalized step width such that $\zeta[0]=0$. Indeed,
in our setting \eqref{eq:P} (i.e.,~$m=1$),
although $\phi_+$ is not differentiable at $[\phi=0]$, thanks to  
$\zeta:U_{\rm ad}\to L^{\infty}(\Omega)$, the sensitivity with the weight can be denoted by $\zeta[\phi_i]\partial_\phi\mathcal{E}((\phi_i)_+)=-\zeta[\phi_i](\beta-\alpha)\chi_{\phi_i}\nabla u_{\phi_i}\cdot \nabla v_{\phi_i}$ formally.

Now, as in \cite{O23}, we characterized the level set function by a solution to the time discrete version of the following doubly nonlinear diffusion equation \cite{OY23}\/{\rm :}  
\begin{align}\label{eq:NLD}
    \partial_t\phi^{q}-\e \Delta_p \phi=
    (\beta-\alpha)\chi_{\phi}\nabla u_{\phi}\cdot\nabla v_{\phi}
        \quad \text{ in } \Omega\times (0,+\infty).
\end{align}
If we set $q\in (0,1)$ such that $q\approx 1$ for simplicity of linearization, one has $|\phi_{i+1}|^{q-1}\approx |\phi_{i}|^{q-1}$, and then the time discrete equation of \eqref{eq:NLD} is described as follows\/{\rm :}  
\begin{align*}
|\phi_{i}|^{q-1}\frac{\phi_{i+1}-\phi_i}{\tau}-\e \Delta_p \phi_{i+1}=
    (\beta-\alpha)\chi_{\phi_i}\nabla u_{\phi_i}\cdot\nabla v_{\phi_i}
        \quad \text{ in } \Omega.
\end{align*}
In particular, multiplying it by $\tau|\phi_{i}|^{1-q}$, we obtain \eqref{eq:NLD2} with $\zeta[\phi_i]=\tau|\phi_{i}|^{1-q}$. 
In this paper, we choose $p=2$ since one expects that the positive parts of optimal level set functions belong to $V$ from Remark \ref{R:H1}.
Thus $\phi_{i+1}$ satisfies
\begin{align}
&\int_\Omega   |\phi_i|^{q-1}(x)\frac{\phi_{i+1}(x)-\phi_{i}(x)}{\tau}\varphi(x)\, \d x
+
\e\int_\Omega \nabla \phi_{i+1}(x)\cdot \nabla \varphi(x)\, \d x\nonumber\\
&\quad =\int_\Omega (\beta-\alpha)\chi_{\phi_i}(x)\nabla u_{\phi_i}(x)\cdot \nabla v_{\phi_i}(x)\varphi(x)\, \d x
\quad \text{ for all }\ \varphi\in V\cap L^{\infty}(\Omega). 
\label{eq:discnld}
\end{align}
As in Algorithm \ref{alg1}, the following algorithm is proposed\/{\rm :}
\begin{algorithm}[H]
    \caption{Optimization for the level set function.}
    \label{alg2}
    \begin{algorithmic}[1]
    \STATE 
    Let $i=0$.
    Set $\Omega\subset \R^d$, $\alpha,\beta, \gamma>0$, $f\in L^2(\Omega;\R_+)$ and $\phi_0\in U_{\rm ad}$ for $p=2$. 
    \STATE
    Solve \eqref{eq:rE} with  $\kappa=k[(\phi_i)_+]$ to determine $u_{\phi_{i}}$ in \eqref{eq:discnld}. 
    \STATE
    Solve \eqref{eq:adE} with $\kappa=k[(\phi_i)_+]$ to determine $v_{\phi_{i}}$ in \eqref{eq:discnld}. 
    \STATE
    Compute \eqref{eq:discnld}.  
  \STATE
Determine $\lambda\in\R$ such that 
$$
|\gamma|\Omega|-\|(\phi_{i+1}^\lambda)_+\|_{L^1(\Omega)}|\le \eta_1, 
$$
where $\eta_1>0$, $\phi^\lambda_{i+1}$ is such that
$$
\phi^\lambda_{i+1}=\max\{-1,\min\{\phi_{i+1}+\lambda,1\}\}.
$$
\STATE
   Check for the convergence condition, 
   $$
   \|\phi^\lambda_{i+1}-\phi_{i}\|_{L^1(\Omega)}\le \eta_2,
   $$ 
   where $\eta_2>0$.
   If it is satisfied, then terminate the optimization as $\phi_{i+1} \leftarrow \phi^\lambda_{i+1}$; otherwise, return 2 after setting $\phi_{i} \leftarrow \phi^\lambda_{i+1}$. 
    \end{algorithmic}
\end{algorithm}

\section{Numerical results for \eqref{eq:P}}
\label{s:numerical-results}
Based on the previous sections, we shall numerically construct the material distribution of two materials with diffusion coefficients of $\alpha>0$ and $\beta>0$ such that the Dirichlet energy is minimized by using FreeFEM++ \cite{H12} with piecewise linear Lagrange elements on a triangular mesh. 
Throughout this section, we set $\boldsymbol{\sigma}=1$, $d=2$, $\Omega=(0,1)^2$, $\alpha=1$, $\beta=10$ and $\theta_0\equiv\phi_0\equiv\gamma$ for $\gamma\in (0,1)$.

\subsection{Numerical validity}
We first check the numerical validity. 
Based on Algorithms \ref{alg1} and \ref{alg2}, we set $\gamma=0.6$ and $f=0.001$.
As for \eqref{eq:discnld}, the characteristic function $\chi_{\phi_i}$ is treated approximately as $0.5\tanh (\phi_{i}/0.1)+0.5$. 
Then we obtain Figures \ref{fig:5-1} and \ref{fig:5-2}. From Figures \ref{valid-1}--\ref{valid-3}, it is confirmed that Algorithm \ref{alg2} makes $\Omega\subset \R^2$ that almost consists of materials with diffusion coefficients of $\alpha>0$ (the blue domain) and $\beta$ (the red domain). 
In particular, it is noteworthy that $[\phi_+=1]:=\{x\in \Omega\colon \phi_+(x)=1\}$ and $[\phi_+=0]$ involve $[\theta=1]$ and $[\theta=0]$, respectively. 
Furthermore, Figure \ref{fig:5-2} shows that the convergence value of the Dirichlet energy $\mathcal{E}$ is monotonically decreasing with respect to $\e>0$, which means that 
the necessary condition \eqref{eq:necessary} is satisfied, and 
$\mathcal{E}(\phi_+)$ asymptotically tends to $\mathcal{E}(\theta)$ constructed by the optimized volume fraction.
As a qualitative property of (locally) optimal configurations, it is suggested that the family of optimal configurations contain two-phase configurations, and then we see that the Dirichlet energy decreases by increasing the perimeter of the interface.
This completes the confirmation of the validity of the proposed method (see also Remark \ref{R:worst} below).

\begin{figure}[htbp]
   \hspace*{-5mm} 
    \begin{tabular}{cccc}
      \begin{minipage}[t]{0.25\hsize}
        \centering
        \includegraphics[keepaspectratio, scale=0.6]{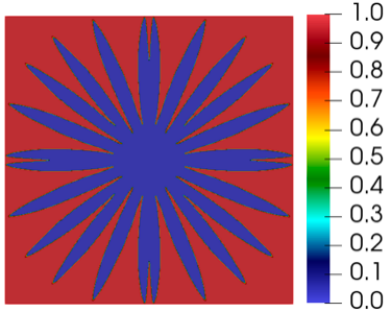}
        \subcaption{$\phi_+$ with $\e=5.0\times 10^{-7}$}
        \label{valid-1}
      \end{minipage} 
      \begin{minipage}[t]{0.25\hsize}
        \centering
        \includegraphics[keepaspectratio, scale=0.6, angle=0]{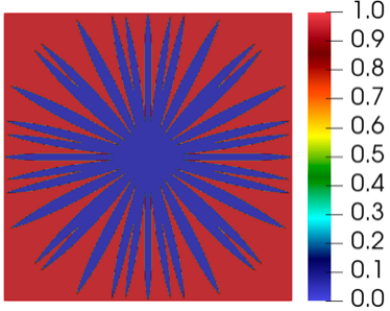}
        \subcaption{$\phi_+$ with $\e=10^{-7}$}
        \label{valid-2}
      \end{minipage} 
         \begin{minipage}[t]{0.25\hsize}
        \centering
        \includegraphics[keepaspectratio, scale=0.6, angle=0]{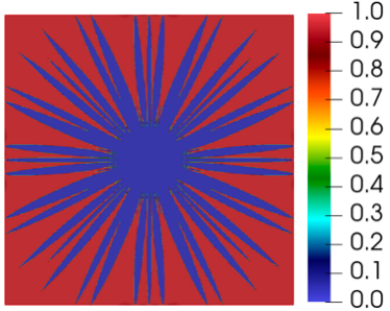}
        \subcaption{$\phi_+$ with $\e=10^{-8}$}
        \label{valid-3}
      \end{minipage}
           \begin{minipage}[t]{0.25\hsize}
        \centering
        \includegraphics[keepaspectratio, scale=0.58, angle=0]{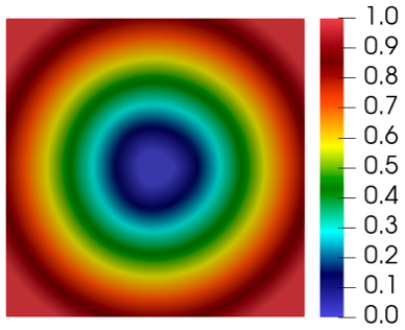}
        \subcaption{$\theta$}
        \label{valid-4}
      \end{minipage}        
    \end{tabular}
    \caption{Optimized configurations. The blue and red domains in (a)--(c) represent materials with diffusion coefficients of $\alpha$ and $\beta$ ($\alpha<\beta$), respectively.}
    \label{fig:5-1}
\end{figure}

\begin{figure}[htbp]
        \centering
        \includegraphics[keepaspectratio, scale=0.33]{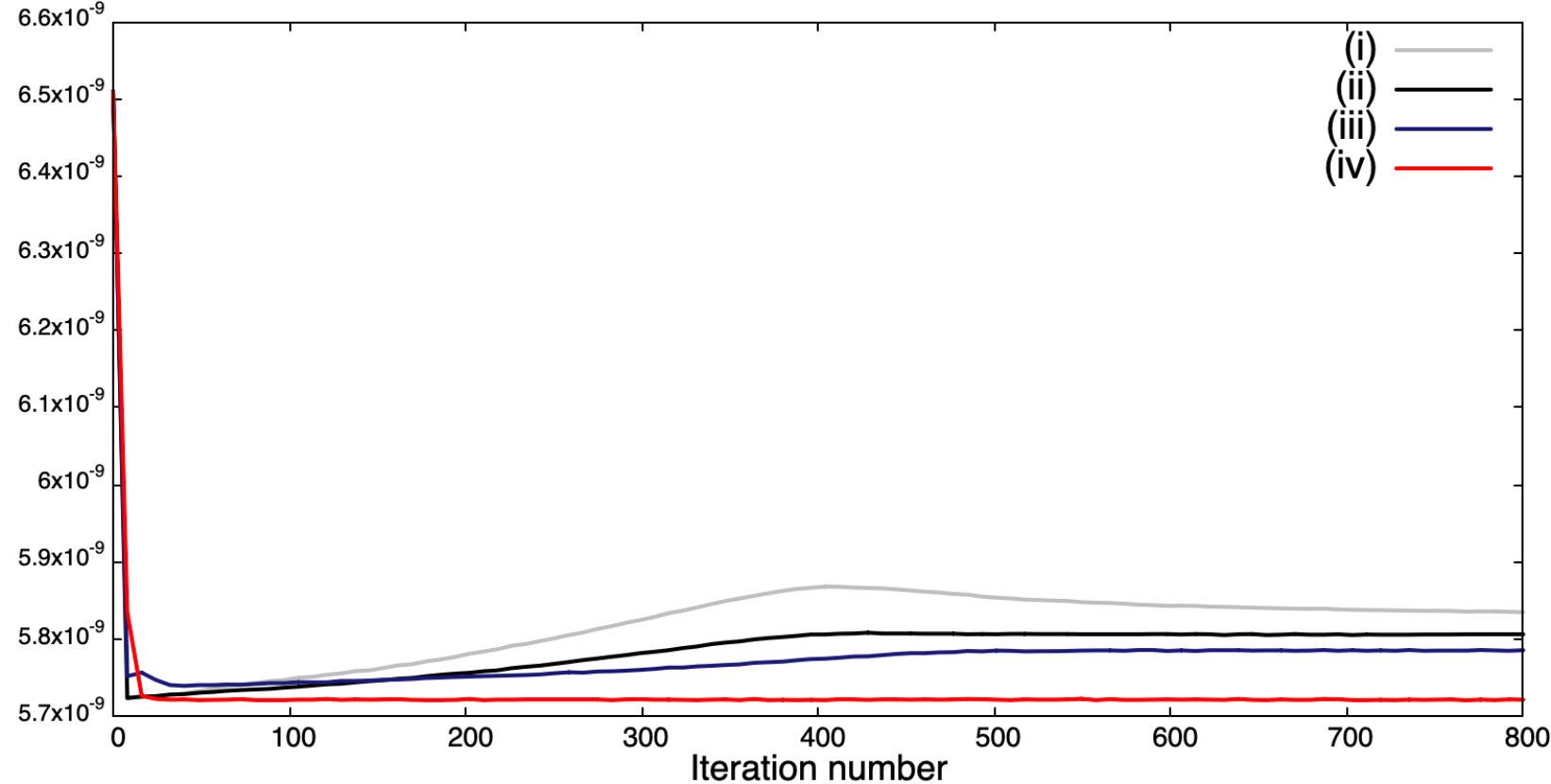}
   \caption{Convergence histories for Dirichlet energies\/{\rm :} (i) $\mathcal{E}((\phi_i)_+)$ with $\e=5.0\times10^{-7}$ (ii) $\mathcal{E}((\phi_i)_+)$ with $\e=10^{-7}$ (iii) $\mathcal{E}((\phi_i)_+)$ with $\e=10^{-8}$ (iv) $\mathcal{E}(\theta_i)$. Here we put $\eta_1=1.0\times 10^{-4}$ and $\eta_2=1.0\times 10^{-5}$.
}
\label{fig:5-2}
\end{figure}

\begin{rmk}[Worst conductor] \label{R:worst}
    \rm 
    The minimization problem \eqref{eq:OP} corresponds to the problem for determining the so-called best two-material thermal conductor. Conversely, as for the worst case, we obtain Figure \ref{fig:worst}. 
    Here we set the objective functional and the diffusion coefficient as $-\mathcal{E} (\chi_{\Omega_1})$ and $\kappa[\chi_{\Omega_1}]=\alpha \chi_{\Omega_1}+(1-\chi_{\Omega_1})\beta$, respectively. Comparing Figure \ref{worst-1} with Figure \ref{worst-2}, we see that similar configurations are obtained, and moreover, it can be confirmed that the convergence values are almost equivalent in Figure \ref{worst-g}. These results suggest the effectiveness of the proposed method.  
\end{rmk}

\begin{figure}[htbp]
   \hspace*{-5mm} 
    \begin{tabular}{cccc}
      \begin{minipage}[t]{0.25\hsize}
        \centering
        \includegraphics[keepaspectratio, scale=0.6]{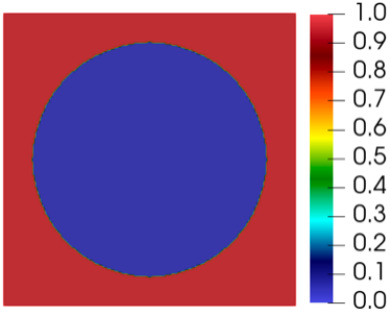}
        \subcaption{$\phi_+$ with $\e=10^{-5}$}
        \label{worst-1}
      \end{minipage} 
      \begin{minipage}[t]{0.25\hsize}
        \centering
        \includegraphics[keepaspectratio, scale=0.6, angle=0]{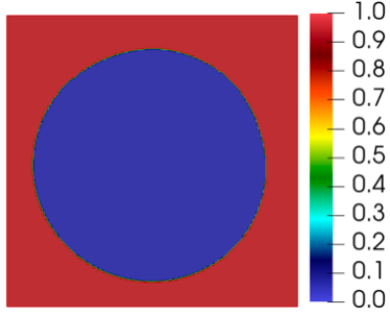}
        \subcaption{$\theta$}
        \label{worst-2}
      \end{minipage} 
         \begin{minipage}[t]{0.5\hsize}
        \centering
        \includegraphics[keepaspectratio, scale=0.55, angle=0]{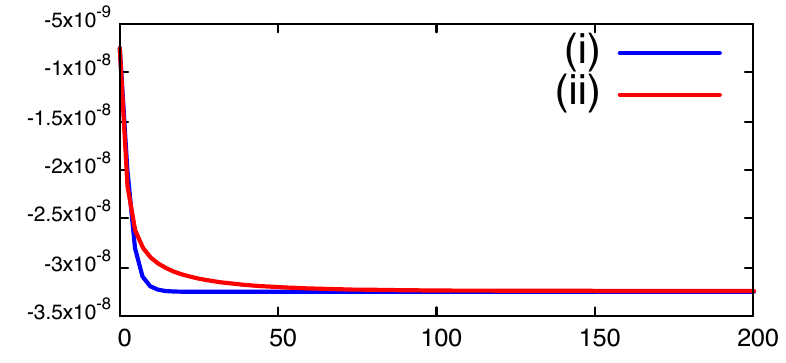}
        \subcaption{Convergence histories\/{\rm :} (i) $-\mathcal{E}(\phi_+)$ with $\e=10^{-5}$ (ii) $-\mathcal{E}(\theta)$. The horizontal axis indicates the iteration number.}
        \label{worst-g}
      \end{minipage}
        \end{tabular}
    \caption{Optimized results in Remark \ref{R:worst}.}
    \label{fig:worst}
\end{figure}

\subsection{Characteristics of nonlinear problems}\label{S:5-2}


This subsection focuses on how the optimized configurations vary for different heat sources to see the characteristics of nonlinear problems.
Here we set $\e=10^{-6}$ and $\gamma=0.5$.
In particular, we compare the cases with the thermal radiation and the Robin boundary conditions. 
Since \eqref{eq:gE} with the Robin boundary condition is linear, the solution is a constant multiple of the original if the heat source is multiplied by a constant, and therefore, 
optimized configurations do not vary by multiplied by a constant of the heat source. 
This is confirmed by Figure \ref{fig:robin}. 
In contrast, with the radiation boundary condition, the solution is different from the original solution multiplied by a constant even if the heat source is multiplied by a constant, and therefore, we see by Figure \ref{fig:radiation} that optimized configurations deeply depend on variations in heat sources, which implies that one of the characteristics of nonlinear problems can be obtained.
Furthermore, based on Figure \ref{fig:comparison}, the following physical interpretation can be made\/{\rm :}

\begin{itemize}
\item[(i)]
 Since both solutions with convection and with radiation asymptotically reach the trivial solution as $|f|\to 0_+$, similar optimized configurations are obtained
 in the case where the value of the heat source is small,
 and therefore, the convergence values of the objective functionals are almost the same; in other words, the contributions of convection and radiation to minimize the energy are almost the same (see (a)).
\item[(ii)]
On the other hand, it can be confirmed that, in the process of increasing the value of the heat source, the optimized configurations with radiation asymptotically tend to be the same as those with the homogeneous Dirichlet boundary condition (see, e.g.,~\cite{ACMOY19,O23}). 
In some cases, convection seems to contribute more to energy minimization than radiation (see (b)). 
\item[(iii)]
Figures (c)--(d) suggest that the contribution of radiation to energy minimization increases with increasing temperature.
\end{itemize}

\begin{figure}[H]
   \hspace*{-5mm} 
    \begin{tabular}{cccc}
      \begin{minipage}[t]{0.25\hsize}
        \centering
        \includegraphics[keepaspectratio, scale=0.6]{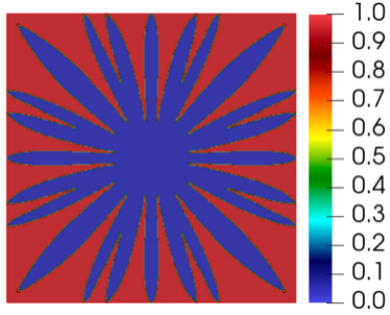}
        \subcaption{$\phi_+$ with $f\equiv 1.0$}
        \label{robin-1}
      \end{minipage} 
      \begin{minipage}[t]{0.25\hsize}
        \centering
        \includegraphics[keepaspectratio, scale=0.6, angle=0]{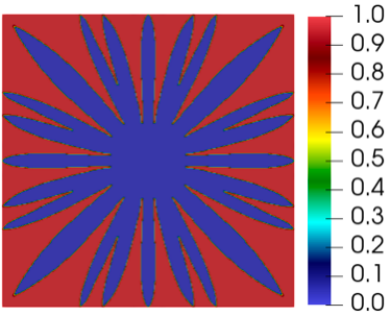}
        \subcaption{$\phi_+$ with $f\equiv 10$}
        \label{robin-2}
      \end{minipage} 
         \begin{minipage}[t]{0.25\hsize}
        \centering
        \includegraphics[keepaspectratio, scale=0.59, angle=0]{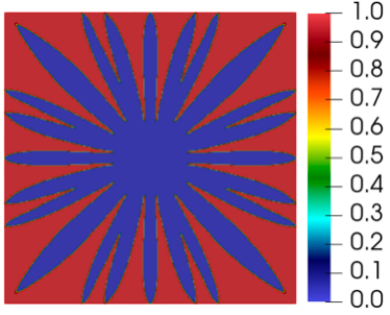}
        \subcaption{$\phi_+$ with $f\equiv 10^2$}
        \label{robin-3}
      \end{minipage}
           \begin{minipage}[t]{0.25\hsize}
        \centering
        \includegraphics[keepaspectratio, scale=0.6, angle=0]{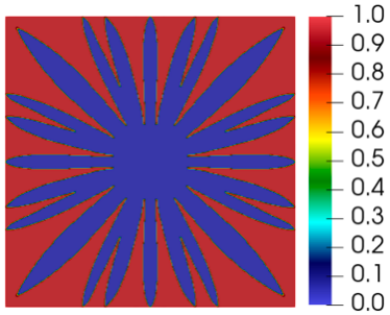}
        \subcaption{$\phi_+$ with $f\equiv 10^7$}
        \label{robin-4}
      \end{minipage}        
    \end{tabular}
    \caption{Optimized configurations for Robin boundary conditions in \S \ref{S:5-2}.}
    \label{fig:robin}
\end{figure}
\begin{figure}[H]
   \hspace*{-5mm} 
    \begin{tabular}{cccc}
      \begin{minipage}[t]{0.25\hsize}
        \centering
        \includegraphics[keepaspectratio, scale=0.6]{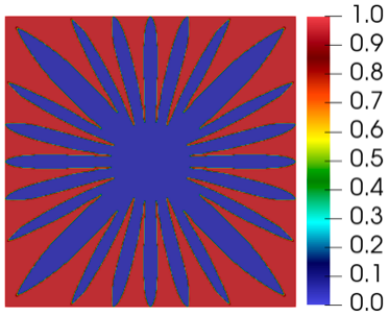}
        \subcaption{$\phi_+$ with $f\equiv 1.0$}
        \label{source-1}
      \end{minipage} 
      \begin{minipage}[t]{0.25\hsize}
        \centering
        \includegraphics[keepaspectratio, scale=0.6, angle=0]{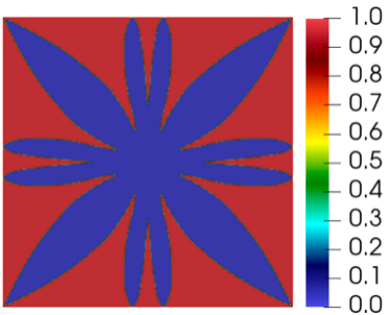}
        \subcaption{$\phi_+$ with $f\equiv 10$}
        \label{source-2}
      \end{minipage} 
         \begin{minipage}[t]{0.25\hsize}
        \centering
        \includegraphics[keepaspectratio, scale=0.59, angle=0]{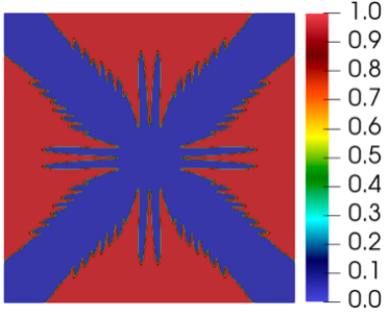}
        \subcaption{$\phi_+$ with $f\equiv 10^2$}
        \label{source-3}
      \end{minipage}
           \begin{minipage}[t]{0.25\hsize}
        \centering
        \includegraphics[keepaspectratio, scale=0.6, angle=0]{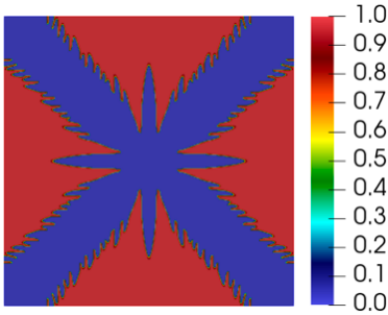}
        \subcaption{$\phi_+$ with $f\equiv 10^7$}
        \label{source-4}
      \end{minipage}        
    \end{tabular}
    \caption{Optimized configurations for radiation boundary conditions in \S \ref{S:5-2}.}
    \label{fig:radiation}
\end{figure}
\begin{figure}[H]
   \hspace*{-5mm} 
    \begin{tabular}{cccc}
      \begin{minipage}[t]{0.5\hsize}
        \centering
        \includegraphics[keepaspectratio, scale=0.37]{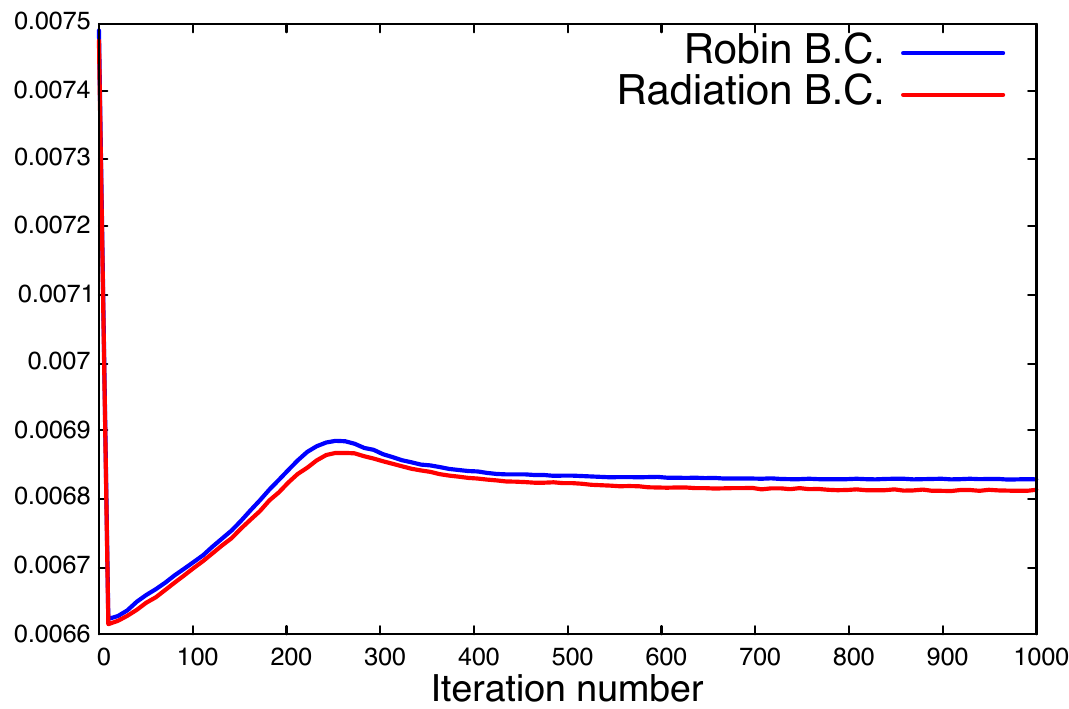}
        \subcaption{$\mathcal{E}((\phi_i)_+)$ with $f\equiv 1.0$.}
        \label{f=1}
      \end{minipage} 
      \begin{minipage}[t]{0.5\hsize}
        \centering
        \includegraphics[keepaspectratio, scale=0.37, angle=0]{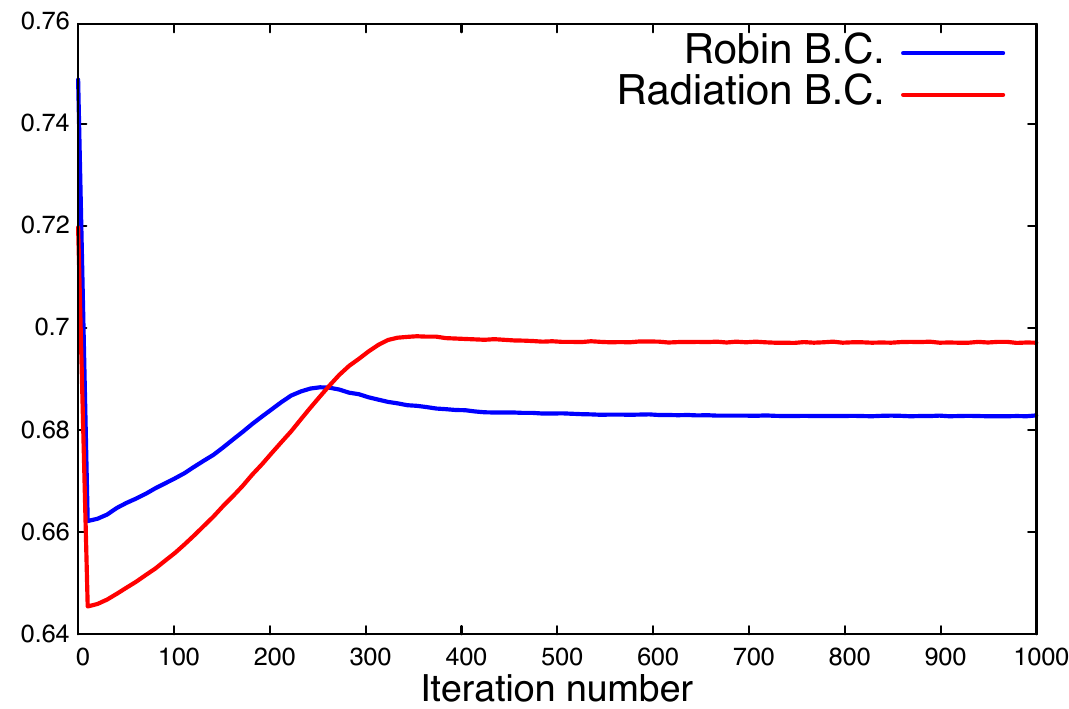}
        \subcaption{$\mathcal{E}((\phi_i)_+)$ with $f\equiv 10$.}
        \label{f=10}
      \end{minipage} \\
         \begin{minipage}[t]{0.5\hsize}
        \centering
        \includegraphics[keepaspectratio, scale=0.37, angle=0]{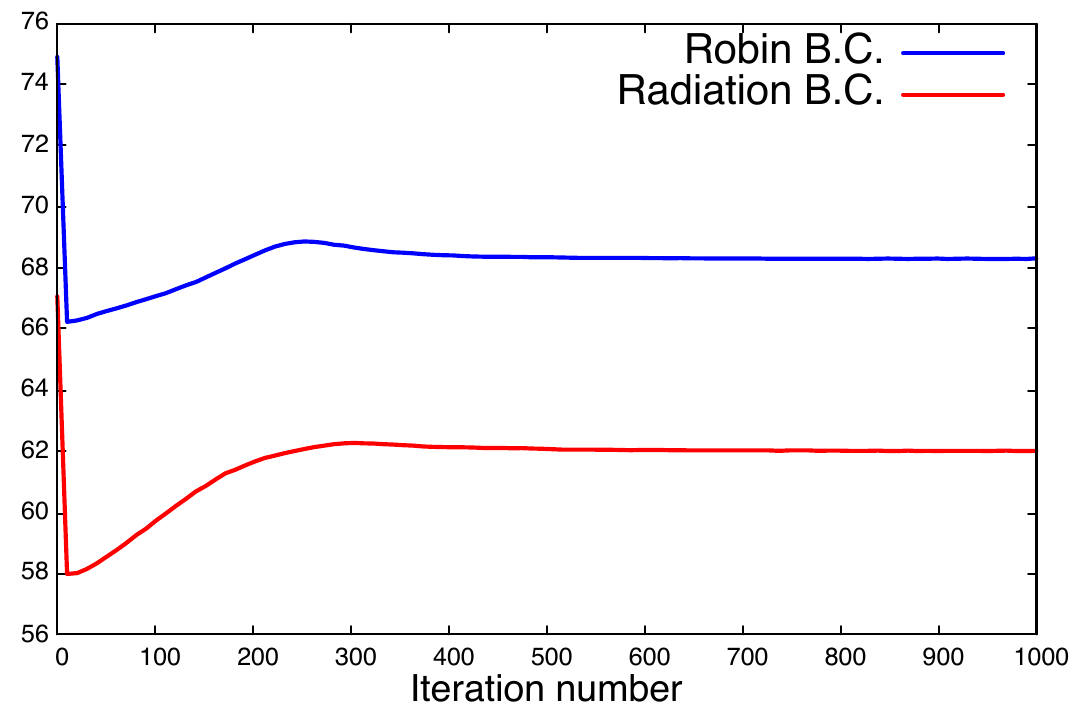}
        \subcaption{$\mathcal{E}((\phi_i)_+)$ with $f\equiv 10^2$.}
        \label{f=100}
      \end{minipage}
           \begin{minipage}[t]{0.5\hsize}
        \centering
        \includegraphics[keepaspectratio, scale=0.37, angle=0]{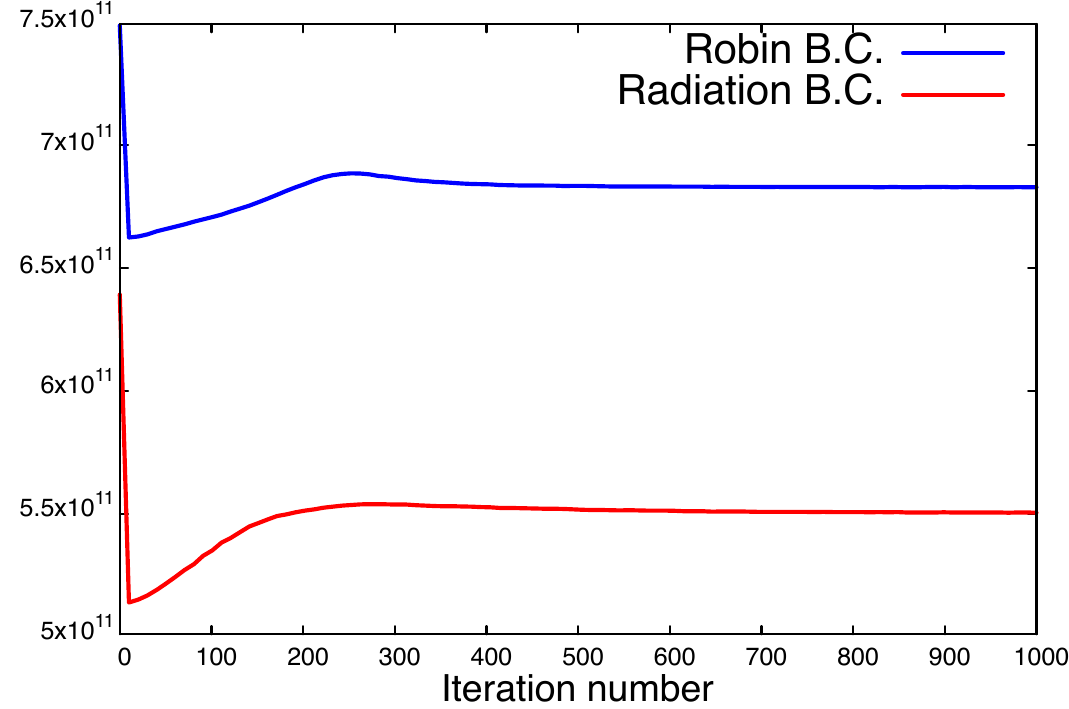}
        \subcaption{$\mathcal{E}((\phi_i)_+)$ with $f\equiv 10^7$.}
        \label{f=10000000}
      \end{minipage}        
    \end{tabular}
    \caption{Convergence histories for the Dirichlet energy $\mathcal{E}((\phi_i)_+)$ comparing Robin and radiation boundary conditions in \S \ref{S:5-2}.}
    \label{fig:comparison}
\end{figure}

\section{Application to thermal radiation problems}
This section is devoted to the application of the optimal design theory and level set-based optimization algorithm to some practical engineering design problems. As the nonlinear boundary condition under consideration describes the thermal radiation, one of the most straightforward but yet important applications is the design of heat radiators. 

Here we assume that a two-phase heat conductor occupying the domain $\Omega$ is situated in a vacuum. The domain $\Omega$ is convex so that its view factor is zero, i.e., no radiating waves can hit the surface $\partial \Omega$. Our aim is to find a piecewise-constant distribution of the coefficient $\kappa$ in $\Omega$ such that it efficiently emits the thermal energy into the ambient vacuum.
\begin{figure}[]
    \centering
    \includegraphics[scale=0.25 ]{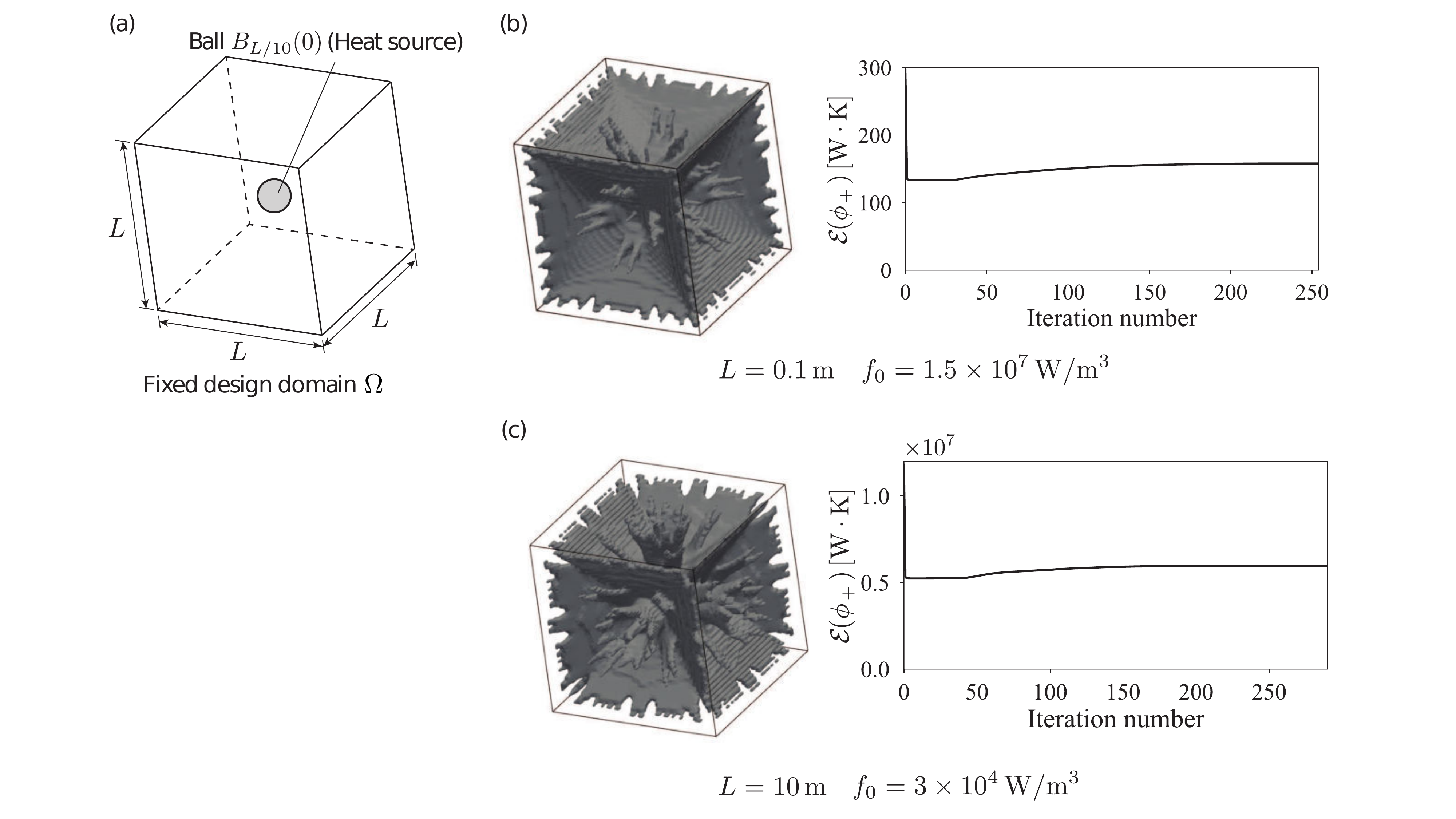}
    \caption{Optimization of three-dimensional heat radiators. The fixed design domain $\Omega$ is set to be a cube $(-L/2,L/2)^3$ with length $L>0$. The heat source is given by $f=f_0 \chi_{B_{L/10}(0)}$ with constant $f_0>0$ as shown in (a). Optimized results are shown in (b) and (c). The shaded regions represent $[\phi>0]$, i.e., conductivity $\beta$.}
    \label{fig:result-case3_result-case5_result}
\end{figure}

As shown in Figure \ref{fig:result-case3_result-case5_result} (a), let $\Omega = (-L/2,L/2)^3$ be the cube with side length $L>0$. The cube contains a heat source $f$.
The boundary of the cube $\Omega$ comprises a radiative surface $\Gamma_R$ and thermally insulated one $\partial\Omega\setminus\Gamma_R$.
Then the temperature $u$ in $\Omega$ solves
\begin{align}
    \begin{cases}
        -\dv(\kappa \nabla u)=f\quad &\text{ in } \Omega,
        \\
        -\kappa \nabla u\cdot \nu=\boldsymbol{\sigma}u^4 \quad &\text{ on } \Gamma_\mathrm{R},
        \\
        -\kappa \nabla u\cdot \nu=0 \quad &\text{ on } \partial\Omega\setminus\Gamma_\mathrm{R}.
\end{cases}
\end{align}

As in Section \ref{s:numerical-results}, we seek the distribution of diffusion coefficients (thermal conductivities) $\alpha$ and $\beta$ such that the Dirichlet (internal) energy is minimized under the volume constraint. Throughout this section, the thermal conductivities are set as $\alpha = 15\, \mathrm{W\,m^{-1}\,K^{-1}}$ (nichrome) and $\beta=400\, \mathrm{W\,m^{-1}\,K^{-1}}$ (copper), respectively.

Let us start with the case of $\Gamma_\mathrm{R} = \partial\Omega$, i.e., all the surfaces are radiative. In this numerical experiment, the volume constraint is set to $\gamma = 0.15$, and the heat source $f$ is uniformly distributed in the ball of radius $L/10$ located at the center of $\Omega$, i.e., $f= f_0 \chi_{B_{L/10}(0)}$ with positive constant $f_0>0$.

Unlike usual conductivity problems with linear boundary conditions, the constants $L$ and $f_0$ may affect the optimizer of the best-conductor problem. This can be confirmed from the results in Figure \ref{fig:result-case3_result-case5_result}, where the optimal designs are shown for two parameter pairs $(L,f_0)$. Figure \ref{fig:result-case3_result-case5_result} (b) and (c) show the optimized configuration of the conductivity $\kappa$ and convergence history of the objective functional. From the results, we observe some clear differences between the two shapes, e.g.,~the number of spikes. 
    \begin{figure}[]
    \centering
    \includegraphics[scale=0.6]{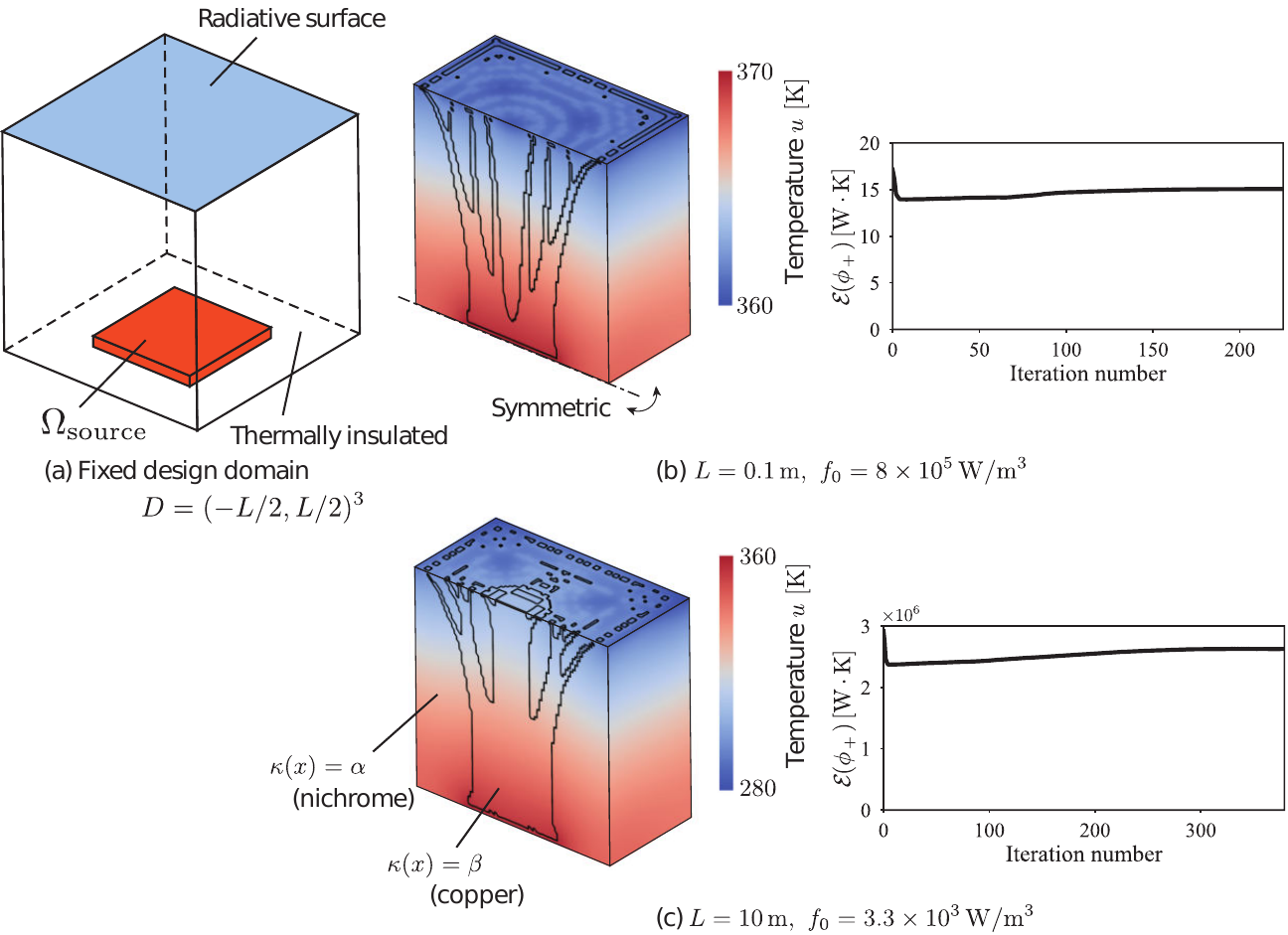}
    \caption{Optimal design of a heat radiator with non-uniformly distributed heat source. (a) Radiative surface and heat source in the fixed design domain $\Omega$. (b) and (c) Optimal design and corresponding temperature field with objective functional values for each parameter pair.}
    \label{fig:result-bottom}
\end{figure}
Another numerical example is shown in Figure \ref{fig:result-bottom}. As in the previous example, the fixed design domain is a cube $\Omega=(-L/2,L/2)^3$ with the length $L>0$. We give a heat source $f$ in the bottom part of $\Omega$ as $f=f_0\chi_{\Omega_\mathrm{source}}$ with constant $f_0>0$ and $\Omega_\mathrm{source} = (-L/4,L/4)\times (-L/4,L/4)\times (-L/2,-9L/20)$. Unlike the previous example, the radiative surface is set to $\Gamma_\mathrm{R} = \{ x_3 = L/2 : x\in\partial\Omega\}$, i.e.,~only the upper surface is radiative. In terms of physics, our aim is to enhance the radiation of heat energy generated in the bulk.

As in the previous example, we consider two cases: $(L,f_0) = (0.1\,\mathrm m, 8\times 10^5\, \mathrm{W/m^3})$ and $(L,f_0) = (10\,\mathrm m, 3.3\times 10^3\, \mathrm{W/m^3})$. The optimized configurations and temperature fields are shown in \ref{fig:result-bottom} (b) and (c) along with convergence history of the objective functional. While both results attain convergence, the obtained design and corresponding temperature fields are quite distinct. The significant difference originates from the nonlinearity in terms of $u$. To see this, let us consider the energies
\begin{align*}
    E_\mathrm{in} =& \mathcal E(\phi_+),
    \\
    E_\mathrm{rad} =& \boldsymbol{\sigma} \int_{\partial\Omega} u^5 (x)\,\mathrm{d}\sigma,
    \\
    E_\mathrm{source} =& \int_\Omega f(x)u(x)\, \d x.
\end{align*}
From the weak form, it immediately follows the energy conservation $E_\mathrm{in} + E_\mathrm{rad} = E_\mathrm{source}$. From a physical point of view, this ratio $E_\mathrm{in} / E_\mathrm{source}$ represents the amount of internal energy stored inside the structure. We calculated the ratio $E_\mathrm{in} / E_\mathrm{source}$ for the two optimized designs (Figure \ref{fig:result-bottom} (b) and (c)) and obtained the values $4.11\times 10^{-3}$ and $0.179$, respectively. This indicates that the usual scaling law does not hold due to the thermal radiation effect.

\begin{figure}[]
    \centering
    \includegraphics[scale=0.55]{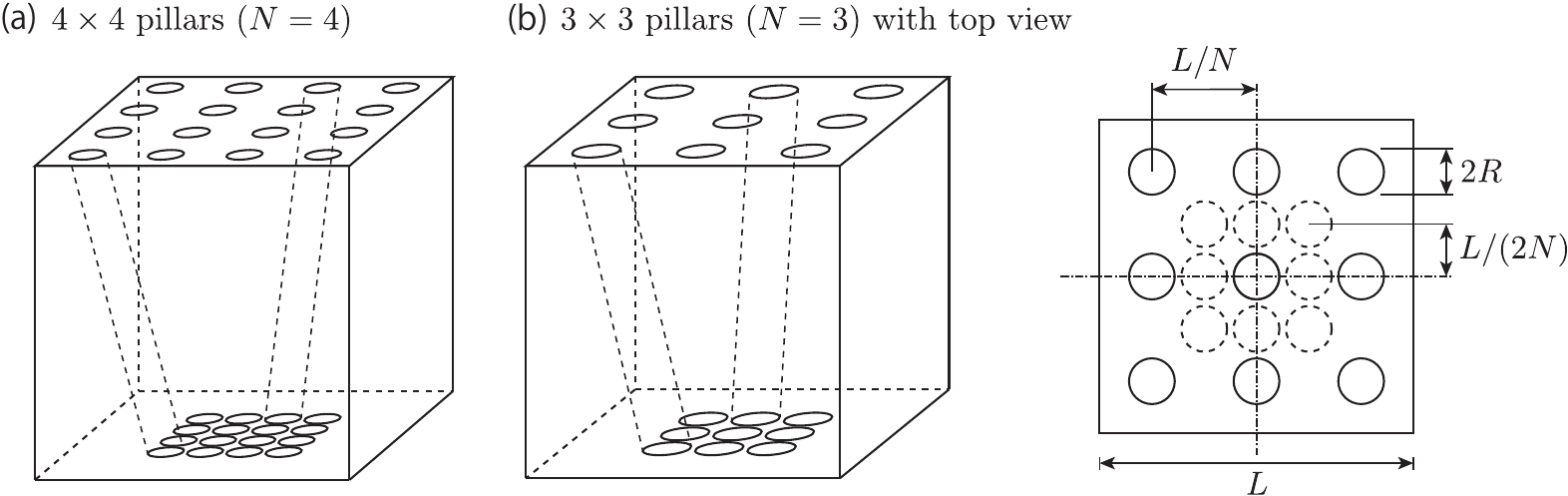}
    \caption{Fin-like structure consisting of $N\times N$ tilted cylinders inside the cube $D$. The radius $R$ of each cylinder is determined such that the total volume of the cylinders is equal to $\gamma L^3$, i.e., $R=\frac{L}{N}\sqrt{\gamma/\pi}$. Two examples are shown in (a) and (b).}
    \label{fig:result-array}
\end{figure}
We finally discuss the performance of the designed heat radiators. As shown in Figure \ref{fig:result-array}, let us consider a fin-like structure with $N\times N$ tilted pillars (cylinders) inside the cube $\Omega=(-L/2,L/2)^3$. The conductivity takes $\kappa=\beta$ inside the pillars and $\kappa=\alpha$ elsewhere. 
The fin-like structure is a reasonable design of a heat radiator as it conducts heat from the bottom to the top radiative surface via the highly conductive pillars. We wish to check that the optimal design is superior to this non-optimized radiator in terms of the objective functional. 

\begin{table}[]
\caption{Values of the energy $\int_\Omega \kappa(x)|\nabla u(x)|^2\mathrm dx \ \mathrm{[W/m^3]}$ for the fin-like structure shown in Figure \ref{fig:result-array}.}
\label{tab:result-array}
\begin{tabular}{lllllllll}
$N$  & $2$    & $3$  & $4$ & $5$ & $6$ & $7$ & $8$  \\ \hline
$L=0.1\,\text{[m]}$ & $22.1$ & $21.4$ & $21.1$ & $20.9$ & $20.8$ & $20.8$ & $20.7$ \\
$L=10 \,\text{[m]}$ & $3.31\times 10^6$ & $3.20\times 10^6$ & $3.16\times 10^6$ & $3.13\times 10^6$ & $3.12\times 10^6$ & $3.10\times 10^6$ & $3.10\times 10^6$
\end{tabular}
\end{table}
To this end, we calculate the value of the energy $\int_\Omega \kappa(x)|\nabla u(x)|^2\mathrm dx$ for various $N$ using the same finite element analysis with quadratic Lagrange elements on a body-fitted tetrahedral mesh. Note that the radius $R$ of each pillar is determined such that the fin-like design satisfies the same volume constraint with $\gamma=0.15$ for fair comparison. The calculated values are shown in Table \ref{tab:result-array}. The results indicate that the energy decreases as the number of pillars increases in both the cases of $L=0.1\,\mathrm{m}$ and $L=10\,\mathrm{m}$. These values are, however, greater than the optimal values $15.1\,\mathrm{W/m^3}$ and $2.63\times 10^6\,\mathrm{W/m^3}$ for $L=0.1\,\mathrm{m}$ and $L=10\,\mathrm{m}$, respectively. These results suggest that the optimal designs yield much more efficient heat radiation than a physically reasonable but non-optimized radiator.

\section{Conclusion}
In this paper, we considered the optimal design problem for the steady-state diffusion equation with nonlinear boundary conditions described by the maximal monotone operator. 
The main target was to analyze the distribution (or shape and topology) of the two-material composite that minimizes the Dirichlet energy with thermal radiation. 
The results obtained in this paper are as follows\/{\rm :}
\begin{itemize}
\item 
We proved that there exists at least a pair of the optimal volume fraction and the optimal homogenized matrix for a true relaxation problem such that the value of the relaxed Dirichlet energy coincides with the minimum value of the original design problem. To this end, we also proved the existence and uniqueness of the weak solution to the state equation with nonlinear boundary conditions described by the maximal monotone operator and the corresponding homogenization theorem.
\item 
In order to estimate the minimum value of the original design problem, the sensitivity of the relaxed Dirichlet energy was derived rigorously, at least under the smoothness assumptions for the domain and the two-material diffusion coefficient.
\item 
We considered the perimeter constraint problem via the positive part of the level set function as an approximation problem for the relaxation problem and proved the existence of minimizers. In particular, it was shown that the limit of the sequence of minimizers with respect to the perturbation parameter becomes a minimizer of the restricted relaxation problem in the Sobolev space.
\item 
By deriving the so-called weighted sensitivity, the level set function was updated by employing the time-discrete version of the nonlinear diffusion equation, and optimized configurations with almost no intermediate sets were obtained. Furthermore, it was numerically verified that the convergence value of the Dirichlet energy is asymptotic to a minimum value if the perturbation parameter is sufficiently small.
\item 
As one of the characteristics of the nonlinear problem, it was confirmed that the optimized configuration deeply depends on the value of the heat source. In particular, the contribution of radiation to energy minimization seems to increase with increasing temperature.
\item Three-dimensional numerical examples were also provided. We designed the distribution of thermal conductivity such that it minimizes internal energy due to an external heat source. The performance of the designed radiators was tested via comparison with a simple fin-like structure.
\end{itemize}

\section*{Declaration of competing interest}
The authors declare that they have no known competing financial interests or personal relationships that could have appeared to influence the work reported in this paper.

\section*{Data availability}
Data will be made available on request.

\end{document}